\documentclass[12pt]{amsart}
\usepackage{geometry} 
\usepackage{tikz}
\usepackage{extarrows}
\usetikzlibrary{matrix,arrows,decorations.pathmorphing}

\newtheorem{theorem}{Theorem}

\newtheorem{definition}{Definition}
\newtheorem{lemma}{Lemma}
\newtheorem{proposition}{Proposition}
\newtheorem{example}{Example}
\newtheorem{remark}{Remark}
\newtheorem{corollary}{Corollary}

\newcommand{\C}{\mathbb{C}}
\newcommand{\Cb}{\mathbf{C}}

\newcommand{\F}{\mathcal{F}}
\newcommand{\M}{\overline{\mathcal{M}}}
\newcommand{\PC}{\mathbb{P}}
\newcommand{\Ps}{\mathcal{P}}
\newcommand{\Q}{\mathbb{Q}}

\newcommand{\R}{\mathbb{R}}
\newcommand{\cS}{\mathcal{S}}
\newcommand{\Ss}{\mathbb{S}}

\title[Modular operad and Galois group]{GENUS ZERO MODULAR OPERAD AND ABSOLUTE GALOIS GROUP}

\author{No\'emie C. Combe, Yuri I. Manin} 
\address{No\'emie C. Combe\\ Max Planck Institute for Mathematics \\ Vivatsgasse 7, 53111 Bonn}
\address{Yuri I. Manin\\ Max Planck Institute for Mathematics \\ Vivatsgasse 7, 53111 Bonn}

\date{} 

\begin{document}

\maketitle

\vspace{.3cm}
\begin{flushright}
\end{flushright}

\vspace{.6cm}

{\bf Abstract:} In this article, we develop the geometry of canonical stratifications of 
the spaces $\M_{0,n}$ and prepare ground for studying the action of the Galois group
$Gal ({\overline{\Q}} /\Q)$
upon strata. We define and introduce a version of a {\it gravity} operad constructed for a class of moduli spaces $\M_{0,n}$, equipped with a hidden holomorphic involution. This additional symmetry is associated to a split quaternionic structure. We introduce a categorical framework to present this object. Interaction between the geometry, physics and the arithmetics are discussed. An important feature is that 0-divisors of the split quaternion algebra imply additional singular points, and lead to investigations concerning the geometry and mixed Hodge structures.

\bigskip


\setcounter{tocdepth}{1}

\tableofcontents

\bigskip
\vfill\eject
\section{Introduction}~\label{S:Intro}

\bigskip

With the development of quantum field theory in theoretical physics,
the interaction of physics and mathematics intensified. As the survey by
M. Atiyah with co--authors~\cite{ADH} stressed, this interaction became 
a very rich source of new ideas in mathematics, in particular, in algebraic geometry.

\smallskip

This article is concerned with one
remarkable fruit of the interaction: creation of the theory of quantum
cohomology (cf.~\cite{KoMa}) and subsequent discovery of its connections
with one of the central objects of number theory, Galois group of the field
of all algebraic numbers (cf. \cite{BrHoRo},~\cite{CoMaMa},~\cite{Ih}, and references therein).

 A quantum cohomology operad has as its components cohomology spaces of moduli spaces $\M_{g,n}$ (or only $g=0$ case), and its structure morphisms correspond to well defined morphisms of these moduli spaces. Choosing for instance \'etale cohomology, the Galois group of algebraic numbers acts upon all operadic components compatibly with operadic morphisms, creating the connection from between theory of quantum cohomology and number theory. 

In this paper, we equip the moduli space $\M_{0,n+1}$ with an algebraic structure. This algebraic structure is motivated by physics and used in quantum theory~\cite{Bae,Va} in general relativity and gravity~\cite{Gog1,KuKoGev,Ul}. These additional algebraic structures imply symmetries.

From a more algebraic perspective, those physical requirements mean that to describe the geometry of the space we need a composition algebra. 
For physical reasons it is not necessary to use field extensions, quadratic extensions over real numbers are enough. 
 
In this paper, we will equip $\M_{0,n+1}$ with a given holomorphic involution $\theta$. The underlying algebra, over which are defined the corresponding modules is a composite normed algebra and in particular a split algebra. Here, the hidden symmetry is associated to split quaternion algebra. For an introduction to the relation between linear spaces corresponding to modules over an algebra, we refer to~\cite{Shi,Shu1,Shu2,Ros}. For manifolds defined over algebras, and verifying the Cauchy--Riemann equations, we refer to~\cite{Scha}, p.146. 

We think that working on spaces being realisations of modules over these split algebras offers an interesting model to consider that encodes also physical data. We introduce this construction in order to define what we call a NY-gravity operad for this model and investigate the geometry around it. Correction to global comment in beginning of section 1 of referee's remarks.

\smallskip

Below we briefly describe the content of our article and its place in a much wider
environment.

\bigskip

{\bf 1.1. Geometric approaches to the absolute Galois group.} 
The study of the absolute Galois group $Gal (\overline{\Q} /\Q)$
and its generalisations, Galois groups of the algebraic closures of various finitely generated
fields, holds a central place in algebraic geometry and number theory. Since the
works of A. Belyi, V. Drinfeld, Y. Ihara, A. Grothendieck, the main approach to
it consists in considering the tower of finite coverings of $\mathbb{P}^1$ ramified only over 
a fixed three points, say, $\{0,1,\infty \}$, and considering actions of two different
groups upon it: geometric one $\pi_1 (\mathbb{P}^1 \setminus \{0,1,\infty \})$,
and algebraic one, $Gal (\overline{\Q} /\Q)$, existing in view of Belyi's theorem.

The latter formed a key step allowing to bridge algebraic curves over $\Q$ and more combinatorial objects (the Grothendieck {\it dessins d'enfants}), by showing that any algebraic curve $X$ over $\Q$, admits a map $X \to \mathbb{P}^1$,
ramified at three points only, say $\{0, 1,\infty\}$. 

Now, considering the moduli space of genus $g$ curves with $n$ marked points $\M_{g,n}$, defined over $\Q$, its algebraic fundamental group $\pi_1^{et}(\Q\times_{\Q} \M_{g,n})$ has an action of $Gal(\overline{\Q}/\Q)$. This statement follows from a more general one in~\cite{SGA} (SGA 1, Exp. IX, Th\'eor\`eme 6.1):

 Let $V$ be a quasi-compact and quasi-separated scheme over $\Q$. Then, there is a short exact sequence
 \[1\to \pi_1^{alg}(V\otimes_{\Q} \overline{\Q})\to \pi_1^{alg}(V)\to Gal(\overline{\Q}/\Q) \to 1,\]
 of profinite topological groups. 

The goal of this article is to start replacing in this picture the tower of coverings of $\mathbb{P}^1$
by another geometric object: the {\it genus zero modular operad,} or, more precisely, its categorical version,
developed in our recent article~\cite{CoMa}.

To achieve this goal, we must suggest one more object, upon which $Gal ({\overline{\Q}} /\Q)$
acts, in this operadic environment. The simplest version of it consists of all points in 
$\M_{0,n}({\overline{\Q}})$, and respective maps induced by operadic morphisms,
but it is too trivial and too large for useful applications. 

Its much smaller and possibly interesting version consists of moduli points
of all {\it maximally degenerated genus zero stable curves with marked points} (strata of dimension zero).

 More generally, we develop here the geometry of canonical stratifications of 
the spaces $\M_{0,n}$, being indexed by dual graphs of the divisors , and prepare ground for studying the action of the Galois group
upon strata.

\bigskip

{\bf 1.2. Genus zero modular operad and its hidden symmetry.} The components
and composition morphisms of the genus zero
modular operad belong to the category of smooth projective manifolds. A naive way to describe
its $n$--ary component $\M_{0,n+3}$ is this. First, consider the moduli space of marked, pairwise distinct points, on the projective line:
it can be naturally identified with an obvious open subset in $(\mathbb{P}^1)^n$ (which is the complement of the discriminant variety).
Second, construct a compactification of this subset by adding as fibres ``stable'' curves of genus
zero with marked points that can be described as degenerations of the generic stable curve. 

\smallskip

Since we mentioned points $ \{0,1,\infty\}$, we have implicitly introduced in this description
the coordinate $t$ on $\mathbb{P}^1$ mentioned above, and thus we can extend the involution
$t\mapsto 1-t$ to induce it upon $\mathcal{M}_{0,n+3}$, and then to
$\M_{0,n+3}$, for $n\geq 0$.

\smallskip

Actually, in this context there is a better way to define the rigidification
involving $t$. To make explicit the geometry behind it, imagine first $\mathbb{P}^1(\mathbb{C})$
as a topological sphere $S^2$
endowed with one complex structure and three ``equators'' $S^1_j\subset S^2$
in general position. For each $j$, we can introduce a complex coordinate $t_j$
upon $S^2$ identifying $S^1_j$ with a naturally oriented $\mathbb{P}^1(\mathbb{R})$.
Then the whole symmetry group behind this rigidification
 will be generated by $t_j\mapsto 1-t_j$, and later it can be extended to the
whole group of symmetries of the modular operad
 of genus zero. This is what we can call its hidden symmetry or stack symmetry. The same group
 is used in [I] and elsewhere in order to treat the cycle relations in the
 Grothendieck--Teichm\"uller group.

 \smallskip
 
In the main part of this article, we restrict ourselves to the study of this {\it hidden involution}
and the {\it stacky} quotient operad with respect to it.

The total automorphism groups of $\overline{M}_{0,n+1}$ were determined by A. Bruno and M. Mella~\cite{BrMe}, using Kapranov’s description of $\M_{0,n}$ as the closure of a subscheme of a given Hilbert scheme. 
In our paper \cite{CoMa} we show how one can use their results
in order to encode the total structure of symmetries of genus zero modular operad
in two different combinatorial structures.

 \medskip

Our principal results here, Theorems 1, 2, and 3, can be interpreted as showing ways
to various objects upon which the absolute Galois group acts,
such as points of maximum degeneration mentioned in the Abstract above.
The structures connecting various components of all these objects into a unity
are again of operadic type, and important role among them is played by $NY$--operad, taking into account also complex conjugation and split-quaternonic structure. Finally, in the two last sections we investigate the geometry around the singular points of this moduli space with symmetries and mixed Hodge structures. It is particularly interesting since additional singular points arise, on this real realisation, from the 0-divisors of the split quaternion algebra.

 
 

 
 \bigskip
 
 \centerline{\bf Acknowledgements}
 
 \medskip
 
 We are very grateful to the Max Planck Institute for Mathematics,
 for hospitality, support, and excellent working conditions.
 
We are grateful to L. Schneps and P. Lochak who kindly answered our questions regarding
 action of the Grothendieck--Teichm\"uller group.

We thank the anonymous reviewer who carefully read the previous draft of the paper
and suggested several useful corrections and editorial changes that we
took care of. This allowed to highlight ideas that were not explicitly presented in the first version. 


 
\section{The stack category of stable labeled curves}\label{S:Gr}
In this new section, we are brought to the notion of {\it groupoids}, which is particularly well suited to the exposition of our results. 
Groupoids are used here in the context of (pre)stable curves. 
 
A {\it prestable curve} over a scheme $T$ is a flat proper morphism $\pi:C\to T$ whose geometric fibers are reduced one-dimensional schemes with at most ordinary double singular points. Its genus is a locally constant function on $T$: $g(t):=dimH^{1}(C_t,\mathcal{O}_t)$. 
Let $S$ be a finite set. An $S$-pointed prestable curve over $T$ is a family, where $\pi:C\to T$ is a prestable curve and $x_i$ are sections such that for any geometric point $t$ of $T$, we have $x_i(t) \neq x_j(t)$. 

Such a curve is {\it stable} if it is connected and if the normalization of each irreducible component, which has genus zero, carries at least three special points. 
Let $(C,\pi,x_i|i\in S)$ be an $S$-pointed prestable curve. It is stable if and only if automorphism groups of its geometric fibers are finite, and if there are no infinitesimal automorphisms.  

\subsection{Groupoids: general setting}~\label{s:F}
For the reader’s convenience, let us briefly recall the notions of groupoids and stack of groupoids such as presented in \cite{Ma99} Chap. 5 paragraph 3.2. 
Let $\F$ and $\cS$ be two categories and let $b:\F\to \cS$ be a functor. For $F\in Ob(\F)$ such that $b(F)=T$, we will call $F$ a family with base $T$, or a $T$-family.

\subsubsection{Condition for groupoids}

In order to form a groupoid, the data must satisfy the following condition. 

First, for any base $T\in \cS$, any morphism of families over $T$ inducing identity on $T$ must be an isomorphism. Designate by the symbol $\F_T$ this subcategory of $\F$. It will be called the fiber of $\F$ over $T$. 
 
Now, for any morphism between bases $\phi: T_1\to T_2$ there must exist a base change functor $\phi^*:\F_{T_2}\to \F_{T_1}$ such that for a pair of composable morphisms $\phi, \psi$ there exists a functor isomorphism $(\phi\circ \psi)^*\cong \psi^*\circ \phi^*$. The latter must satisfy the cocycle condition expressing associativity.



\subsubsection{1--morphisms of abstract groupoids} 

We will be considering only morphisms between groupoids over the same category of bases $\cS$. Such a morphism $\{b_1:\F_1\to \cS\}\to  \{b_2:\F_2\to \cS\}$ is a functor $\Phi:\F_1 \to \F_2$ such that $b_2\circ \Phi=b_1$. 
\begin{center}
\begin{tikzpicture}[->,>=stealth',shorten >=1pt,auto,node distance=1.5cm,main node/.style={font=\sffamily\bfseries}]
\node[main node] (1) {$\F_1$};
 \node[main node] (2) [right of=1] {};
\node[main node] (3)[right of=2] {$\F_2$};
 \node[main node] (4)[below of=2]{$\cS$};
  
 \path[every node/.style={font=\sffamily\small}]
   (1) edge node[above] {$\Phi$} (3)
    edge node [left] {$b_1$} (4) 
   (3)edge node[right] {$b_2$} (4);
\end{tikzpicture}
\end{center}

\subsection{Groupoids of $S$--labeled stable curves}\,

We consider the categories $\cS=Sch_{\mathbb{Q}}$ of schemes over $\mathbb{Q}$ and of  $S$-labeled stable curves $\F=\overline{\mathcal{M}}_{g,S}$; the genus 0 case is denoted by $\overline{\mathcal{M}}_{S}$. For $T\in\cS$, an object $\F_T$ is a stable $S$--labeled curves over the scheme $T\in \cS$. These $S$-pointed curves are given by $n$-sections $\{s_1,s_2,\dots, s_n\}$. 

A morphism $(C_1/T_1,x_{1i}|i \in S)\to$ $ (C_2/T_2,x_{2i} | i \in S)$ is a pair of compatible morphisms $\phi:T_1\to T_2$ and $\psi:C_1\to C_2$ such that $\psi$ induces an isomorphism of labeled curves $ C_1\to \phi^{*}(C_2)$. Equivalently, the diagram

\begin{center}
\begin{tikzpicture}[->,>=stealth',shorten >=1pt,auto,node distance=1.25cm,main node/.style={font=\sffamily\bfseries}]
\node[main node] (1) {$C_1$};
 \node[main node] (2) [right of=1] {};
\node[main node] (3)[right of=2] {$C_2$};
\node[main node] (4)[below of=1] {$T_1$};
 \node[main node] (5) [right of=4] {};
\node[main node] (6)[right of=5] {$T_2$};  
 \path[every node/.style={font=\sffamily\small}]
   (1) edge node[above] {$\psi$} (3)
   edge node[above] {$$} (4)
   (3)edge node[above] {$$} (6)
   (4)edge node[below] {$\phi$} (6);
\end{tikzpicture}
\end{center}
is cartesian, and induces the bijection of the two families of $S$--labeled sections.

\subsection{Groupoids of universal curves}
Let us consider the groupoid of universal curves $\mathcal{C}_g$.
The objects of $\mathcal{C}_g$ are stable curves $(C/T,x_{i}|i\in S)$, endowed with an additional section $\tilde{\Delta}:T \to C$ not constrained by any restrictions. The morphisms must be compatible with this additional data.

 \subsection{Stacks}
 
\begin{definition}[5.1 in~\cite{Ma99}, 4.1 in~\cite{DM}]~\label{DM}
A stack of groupoids is a quadruple \[(\F,\cS,b:\F\to \cS, \text{Grothendieck topology}~ \mathcal{T}~ on ~\cS)\]
satisfying the following conditions:

\begin{enumerate}
\item $b:\F\to \cS$ is a groupoid (such as defined in section \ref{s:F}). Each contravariant representable functor $S^{op}\to Sets$ is a sheaf on $\mathcal{T}$.
\item Any isomorphism between families over a given base is uniquely defined by its restrictions to the elements of any covering of the base. Given $X_1,X_2$ over $T$, the functor $T'\mapsto Iso_{T'}(X_1\to X_2)$ 
\item Any family over a given base is uniquely defined by its local restrictions. 
\end{enumerate}
\end{definition}

The stacks over $T$ are the objects of a 2--category (stacks/$T$): 1--morphisms are functors from one stack to another, and 2-morphisms are morphisms of functors. In this 2--category every 2--morphism is an isomorphism.

\begin{remark} $\M_{0,S}$ are algebraic Deligne--Mumford (DM) stacks over $\Q$ with respect to some Grothendieck topology (see also paragraph 3.1) and~\cite{DM}. The Grothendieck--Teichm\"uller set-up appears deeply intertwined within the situation where we have added an extra quadratic extension using $\theta$, creating an analog of field extension. \end{remark}

\section{Operads}\label{S:Op}
The notation 
$\mathcal{M}_{g,n}$ stands for pointed curves (with respect to $n$ disjoint sections). 
For general categorical background and stacks we refer to~\cite{KaScha}, Chap. 1, 4, 16, 19.

\smallskip 

 In our line of sight, we aim at constructing a topological operad, obtained by endowing the moduli space $\M_{0,n+1}$ with a symmetry, denoted by $\theta$. The initial motivation for choosing as a symmetry the holomorphic involution $\theta$ was that it does not produce any problems related to orientability. 
However, equipping our space with such a symmetry impacts the underlying algebra.

More precisely, in physics, to express the transition from a manifold (which is a result of some given measurements) to geometry one must be able to introduce a distance between some objects and an etalon for their comparison - an etalon corresponds to a unit in the algebra. In algebraic language all the physical requirements mean that to describe the geometry we need a {\it composition algebra}. 

Although, up to now, the division algebras such as split algebras were mostly investigated in quantum theory \cite{Bae,Va}, we think that working on spaces being realisations of modules over these split algebras offers an interesting model to consider. Therefore, this paper is based on considerations of split algebras and their applications. In particular, here, the hidden symmetry is associated to split quaternion algebra. 

This algebra is a composite, normed algebra with zero divisors. Zero divisors of the split algebras play an important role in the description of mirror symmetries and super symmetries \cite{BaeHu1,BaeHu2, KuTo}. This split quaternion algebra is non-commutative, associative, non-division ring, isomorphic to the ring of $2\times2$ real matrices. This is a model often used for space-time general relativity and gravity. We think it is an interesting model to consider.

We will study in section\, \ref{s:sdm} this hidden symmetry of $\M_{0,n+1}$ in a categorical framework. 
\subsection*{Operad in a category}
Let $({\bf C}, \otimes, 1_{C},a, l, r, \tau)$ be a symmetric monoidal category. Let {\bf Fin} be the category of finite sets with bijections. Given any subset $X\subset Y$, we use the notation $Y/X:= Y\setminus X\sqcup \{*\}$.

\begin{definition}[Operad,~\cite{MaVa} Definition 4.1]\label{D:Op}
An operad $\mathcal{P}$ in {\bf C} is a presheaf $\mathcal{P}: Fin^{op}\to {\bf C}$ endowed with  partial operadic compositions $\circ_{X\subset Y}$ or any $X\subset Y\in Fin$ :
\[\mathcal{P}(Y/X)\otimes \mathcal{P}(X)\to \mathcal{P}(Y),\] for any $X\subset Y$ and a unit $\eta: 1_{{\bf C}}\to \mathcal{P}(\{*\})$ such that the following diagrams commute. 

\vspace{3pt}
\begin{center}
\begin{tikzpicture}[->,>=stealth',shorten >=1pt,auto,node distance=1.87cm,main node/.style={font=\sffamily\bfseries}]
\node[main node] (1) {$\scriptstyle\Ps(Z\setminus Y)\otimes\Ps(Y\setminus X))\otimes \Ps(X)$};
 \node[main node] (2) [right of=1] {};
  \node[main node] (3) [right of=2] {};
\node[main node] (4)[right of=3] {$\scriptstyle \Ps(Z\setminus Y)\otimes(\Ps(Y\setminus X)\otimes \Ps(X))$};
\node[main node] (5)[right of=4] {};
 \node[main node] (6) [right of=5] {};
\node[main node] (7)[right of=6] {$\scriptstyle \Ps(Z\setminus Y)\otimes \Ps(Y))$};  
 \node[main node] (8) [below of=1] { $\scriptstyle\Ps((Z\setminus Y)\setminus (Y\setminus X) \otimes\Ps(Y\setminus X))\otimes\Ps(X))$};
 \node[main node] (9) [below of=4] {$\scriptstyle \Ps(Z\setminus X)\otimes \Ps(X))$};
 \node[main node] (10) [below of=7] {$\scriptstyle\Ps(Z)$};
 \path[every node/.style={font=\sffamily\small}]
   (1) edge node[above] {$\scriptstyle\alpha$} (4)
   edge node[left] {$\scriptstyle\cong$} (8)
   (4)edge node[above] {$\scriptstyle id\otimes\circ_{X\subset Y}$} (7)
   (7)edge node[right] {$\scriptstyle \circ_{Y\subset Z} $} (10)
   (8)edge node[below] {$\scriptstyle \circ_{Y\setminus X\subset Z\setminus X}\otimes id$} (9)
   (9)edge node[below] {$\scriptstyle \circ_{Y\subset Z}$} (10);
\end{tikzpicture}
\end{center}

\vspace{3pt}
\begin{center}
\begin{tikzpicture}[->,>=stealth',shorten >=1pt,auto,node distance=1.8cm,main node/.style={font=\sffamily\bfseries}]
\node[main node] (1) {$\scriptstyle\Ps(Z\setminus X)\setminus Y)\otimes(\Ps(X)\otimes \Ps(Y))$};
 \node[main node] (2) [right of=1] {};
  \node[main node] (3) [right of=2] {};
\node[main node] (4)[right of=3] {$\scriptstyle \Ps((Z\setminus X)\setminus Y))\otimes(\Ps(Y)\otimes \Ps(X))$};
\node[main node] (5)[right of=4] {};
 \node[main node] (6) [right of=5] {};
\node[main node] (7)[right of=6] {$\scriptstyle (\Ps((Z\setminus X)\setminus Y)\otimes \Ps(Y))\otimes\Ps(X)$};  
 \node[main node] (8) [below of=1] { $\scriptstyle \Ps((Z\setminus X)\setminus Y )\otimes\Ps(X)\otimes\Ps(Y)$};
 \node[main node] (9) [below of=7] {$\scriptstyle \Ps(Z\setminus X)\otimes \Ps(X)$};
 \node[main node] (10) [below of=8] {$\scriptstyle(\Ps((Z\setminus Y)\setminus X ) \otimes\Ps(X))\otimes\Ps(Y)$};
 \node[main node] (11) [right of=10] {};
 \node[main node] (12) [right of=11] {};
 \node[main node] (13) [right of=12] {$\scriptstyle \Ps(Z\setminus Y)\otimes \Ps(Y) $};
 \node[main node] (14) [below of=9] {$\scriptstyle \Ps(Z)$};
 \path[every node/.style={font=\sffamily\small}]
   (1) edge node[above] {$\scriptstyle id\otimes \tau$} (4)
   edge node[left] {$\scriptstyle \alpha^{-1}$} (8)
   (4)edge node[above] {$\scriptstyle \alpha^{-1}$} (7)
   (7)edge node[right] {$\scriptstyle \circ_{Y\subset Z\setminus X}\otimes id$} (9)
   (8)edge node[left] {$\scriptstyle \cong $} (10)
    (9)edge node[right] {$\scriptstyle \circ_{X\subset Y} $} (14)
   (10)edge node[below] {$\scriptstyle \circ_{X\subset Z\setminus Y}\otimes id$} (13)
   (13)edge node[below] {$\scriptstyle \circ_{Y\subset Z}$} (14);
\end{tikzpicture}
\end{center}
 \end{definition}

From~\cite{MaVa} we know that the skeletal category of {\bf Fin} is the groupoid $\Ss$, whose objects are $\{1,2,...,n\}$ for any $n\in\mathbb{N}$. Morphisms are the elements of the symmetric groups $\Ss_{n}$. A presheaf on {\bf Fin} is thus equivalent to a collection $\{\mathcal{P}(n)\}_{n\in \mathbb{N}}$ of right $\Ss_n$-modules. In these terms the above structure of operad is equivalent to partial composition products 
$\circ_{i}:\mathcal{P}(n)\otimes \mathcal{P}(m)\to\mathcal{P}(n+m-1)$, for $1\leq i\leq n$, and a unit map $\eta:1_{{\bf C}}\to \mathcal{P}(1)$ satisfying the analogous axioms.

\subsection*{Operad in groupoids}
An operad in groupoids---also known as an operad in the category of small categories---is defined in three steps. Note that it consists of a {\it sequence} of small categories $\mathcal{P}(r), r \in \mathbb{N}$, each of which are equipped with a symmetric group action; a unit morphism $\eta : pt \to \mathcal{P}(1)$, and a composition product $\mu : \mathcal{P}(r)\times \mathcal{P}(n_1)\times \dots \times \mathcal{P}(n_{r}) \to \mathcal{P}(n_1 + \dots + n_{r})$, is formed in the category of categories. Classical identities expressed by diagrams (mainly concerning the associativity, equivariance) hold. 

\hspace{3pt}

More precisely, since the category of groupoids forms a {\it symmetric monoidal subcategory} of the category of small categories, an operad in groupoids can be defined as an operad in categories $\mathcal{P}$, of which components $\mathcal{P}(r)$ are groupoids. The composition structure of an operad in categories (resp. groupoids) can be defined by giving a collection of functors $\circ_{k}:\mathcal{P}(m)\times \mathcal{P}(n)\to \mathcal{P}(m+n-1)$, $k=1,...,m$ satisfying the equivariance, unit and associativity relations.

A morphism of operads in categories $f : \mathcal{P} \to \mathcal{Q}$ is a sequence of functors $f :\mathcal{P}(r) \to \mathcal{Q}(r)$ preserving the internal structures attached to operads. The category of operads in groupoids forms a full subcategory of the category of operads in categories.  
For operads in categories, we will naturally consider the operad morphisms of which all underlying functors $f: \mathcal{P}(r)\to \mathcal{Q}(r)$ are equivalences of categories. 

\subsection{Operads of stable $S$--labeled curves}

For a finite set $S$ denote by $\M_{g,S}$ the Deligne--Mumford stack classifying stable curves of genus $g$ and with marked points $(x_{i})_{i\in S}$ labeled by the finite set $S$. For any injective map $\phi: S \to S'$ of finite sets and any $g$ such that $\M_{g,S}\neq \emptyset$, there is a natural morphism of stacks $\M_{g,S'}\to \M_{g,S}$ called {\it stable forgetting map}, see \cite{Knu2}. 

Our attention focuses on the description of the genus zero modular operad (see~\cite{Man18}, section 4 for details). Take a finite set $S=\{0,1,...,n\}$ and equip it with the action of the symmetric group $\Ss_{n+1}$. Whenever needed the notation $[n]$ will be used to designate the set $\{0,1,...,n\}$.
 
 We have the following structure morphism defined point-wise by a glueing of the respective stable curves:

\begin{equation}\label{E:1}
\M_{0,k+1}\times\M_{0,m_1+1}\times\dots\times\M_{0,m_k+1}\to \M_{0,m_1+\dots +m_k+1}
\end{equation}

\subsection{Hidden symmetry of the DM--Stack}
 \subsubsection{}
 
 Consider the DM--stack $\M_{0,n+1}$. Choose a stable $S$--labeled curve $(C/T, \pi (x_i)\, |\, i\in S,\, |S|=n+1)$, where $T$ is a base scheme in the category $Sch_{\Q}$, $C$ is a scheme in the category $\F$ and $\pi: C\to T$ is a proper flat morphism. Recall, that the marked points are sections over $ t\in T$, where $t$ is a geometric point. The stack $\M_{0,n+1}$ has as $S$-fiber an $n+1$-pointed curve $C\to T$ with $n$ sections $T\to C$ with disjoint images. A section of $C$ is a morphism of $T$-schemes defined from $T$ to $C$ such that composed with $\pi$ one obtains the identity $Id_{T}$.

\smallskip 

We consider the contravariant functor $\M_{0,n+1}$ mapping a scheme $T$ to a collection of $n+1$--pointed curves of genus 0 over $S$. 
 Note that this stack description is the lax functor one; not the one implemented in part 2.
 
Knudsen~(\cite{Knu2}, section 1) shows that $\M_{0,n+1}$ is represented by a smooth complete variety $T_{n}$, together with a universal curve $(C,(x_i)_{\in S}) \to T_{n}$, where $(x_i)_{i\in S}$ are the universal sections. 

\smallskip 

The complement of the  removed *in the literature* ``fat diagonal'' $\Delta$ (i.e. locus of colliding points), is contained in $T_{n}$, as the open subset over which the morphism $\pi$ is smooth. This is the open set parametrizing $n+1$--pointed curves over $Spec(\C)$, for which the curve is $\mathbb{P}^1$. 


Consider the $S\cup \{*\}$--pointed stable curve, with sections $(x_i)_{i\in S\cup \{*\}}$ and having an additional section $x_{*}$. There exists (up to unique isomorphism) a unique $[n+1]$--pointed curve $\pi_{*}:C^{*}\to T_{n+1}$ with sections $(x'_i)_{i\in [n+1]}$ such that $C$ is the contraction of $C^{*}$ along $x'_{[n+1]}$ and such that $x'_{n+1}$ is send to the section $x_{*}$. The pointed curve $C^{*}$ is obtained from $C$ by an explicitly described blow-up, see \cite{Kee}, section 1, for details.
 
The important result is as follows. If $\M_{0,n+1}$ is a category fibered in groupoid $C\to T$, with sections $(x_i)_{i\in S}$, then: $\M_{0,n+2}$ is $\pi: (C, (x_i)_{i\in \{0,...,n+2\}})\to T_{n+1}$, with $ T_{n+1}=(C,(x_i)_{i\in \{0,...,n\}})$ and $(C,(x_i)_{i\in \{0,...,n+2\}})$ is a blow-up of \\$T_{n+1}\times_{T_{n}} T_{n+1}.$
 

\subsubsection{Symmetry of the DM--Stack}\label{s:sdm}
In real algebraic geometry, a classical approach is to endow the moduli space of genus 0 curves and marked points with a antiholomorphic involution $j: \C\to \C$ such that $j^2=Id_{\C}$. This method leads to real algebraic curves. However, using an antiholomorphic involution induces orientability problems, which we want to avoid. 

Instead, we introduce a {\it holomorphic} involution. This construction will lead to a split algebra structure. More precisely, a {\it split quaternionic structure} is considered. It is extensively used in general gravity and also appears in supersymmetry. Recently, this type of algebraic structures, attracted some attention in the domains of relativity and particle theories. Since the vector part of the split quaternions, represents the (2+1)-Minkowski space-time, and not the usual 3-dimensional Euclidean space \cite{Gog1}.

Consider an affine orientable symmetry group \[G=\langle\, \theta| \, \theta^2=Id\, \rangle,\] where for any $x\in \C$, \[\theta: x\mapsto 1-x;\]
and consider the representation of this group as follows. Let $C/T$ be an $n$-pointed stable curve:
 \[G\to \mathrm{Aut}(C/T,(x_{i})).\]


Let us consider the following example and its generalisation.
 Take a disc $\mathbf{D}$. It can be partitioned by a pair of orthogonal lines into 4 identical pieces (the quadrants) denoted $Q_i, (i=1,2,3,4)$. These quadrants are incident at a common vertex. We will denote a pair of adjacent quadrants $Q_{i_1}$ and $Q_{i_2}$ and their respectively diagonally opposite quadrants $Q'_{i_1}$ and $Q'_{i_2}$. 
This vertex is precisely the fixed point under involution.
The group action of $G$ on $\mathbf{D}$ maps any marked point $x(t)$ lying in a given quadrant $Q_i$ to a point $x’(t)$ lying in its diagonally opposite quadrant $Q’_i$.
This example can be easily generalised to an $n$-dimensional sphere $\mathbf{S}^n$. This sphere can, as well, be partitioned by $n$ mutually orthogonal sections into $2^n$ identical $n$-dimensional pieces, 
where $Q_k^n, k= 1,…, 2^{n-1}$ and $Q'^n_k, k=1,…,2^{n-1}$, are the partitioning pieces and $Q'$ are the diagonally opposite pieces to the $Q$ pieces. All the $2^n$ pieces are incident at a common vertex. 
This vertex is the fixed point under the involution.
The group action of $G$ on this sphere $\mathbf{S}^n$ maps any marked point $x (t) \in Q^n_i$ to the point $x’(t)\in Q’^n_i$ where $Q^n$ and $Q’^n$ are mutually diagonally opposite piece.

As a more down to earth approach, we can look at $\theta$ as an affine map acting on the set of $n$-sections as follows: \[A_{\theta}: \mathbb{C}^{n}\to \mathbb{C}^{2n}\]
\[(x_1,...,x_n)\mapsto (x_1,...,x_n,1-x_1,...,1-x_n).\]
In matrix notation one has: 
\[ 
(x_1,...,x_n) \left[\begin{smallmatrix} 
1& 0 & \dots &0 |&-1& 0 & \dots &0\\
0 & 1 & \dots &0|&0& -1 & \dots &0 \\
0& 0 &\dots & 1|&0& 0 & \dots &-1 
\end{smallmatrix}\right]_{(n\times 2n)}+\left(\left[\begin{smallmatrix} 0 & \dots & 0\\ 0 & \dots & 0\\ 0& \dots &0 \end{smallmatrix}\right]_{1\times n}, \left[\begin{smallmatrix} 0 & \dots & 0\\ 0 & \dots & 0\\ 0& \dots &0 \end{smallmatrix}\right]_{1\times n}\right)_{1\times 2n}
\]

Consider the graph function of $\theta\in \mathrm{Aut}(C/T,(x_{i}))$. Given the set of $S$--labeled points on $C/T$, consider a binary relation $\theta$ (endorelation) on $(C/T; (x_{i})_{i\in S})$. This binary relation is given by the holomorphic involution $\theta:x\mapsto 1-x$, which maps the $n$-tuple $(x_i)_{i\in S}$ of labeled marked points on $C/T$ to $(\theta(x_i))_{i\in S}$. The graph of the binary relation $\theta$ from $C/T$ to itself is formed from the pairs $({\bf x}, \theta({\bf x}))$ and the relation is functional and entire. For $(C/T, (x_i,\theta(x_i))_{i\in S})$ we have:

\vspace{3pt}
\begin{center}
\begin{figure}\label{F:dia}
\begin{tikzpicture}[->,>=stealth',shorten >=1pt,auto,node distance=1.8cm,main node/.style={font=\sffamily\bfseries}]
\node[main node] (1) {};
 \node[main node] (2) [right of=1] {};
\node[main node] (3)[right of=2] {$ C_{\{x_1,...,x_n\}\cup \{\theta(x_1),...,\theta(x_n)\}}$};
\node[main node] (4)[right of=3] {};
\node[main node] (5)[right of=4] {};  
 \node[main node] (6) [below of=1] { $ C_{\{x_1,...x_n\}}$};
 \node[main node] (7) [below of=3] {};
 \node[main node] (8) [below of=5] {$ 
 C_{\{\theta(x_1),...,\theta(x_n)\}})$};
 \node[main node] (9) [below of=7] {$T$};
 
 \path[every node/.style={font=\sffamily\small}]
 (3) edge node[above left] {$ \pi_{1}$} (6)
  edge node[above right] {$ \pi_{2}$} (8)
  edge [dashed,bend left=10] node[above] {$$} (9)
  (6) edge[bend left=7] node[above] {$ \theta$} (8)
  edge node[below left] {$ \pi^{-1}$} (9)
   (8)edge node[below right] {$ \pi\circ\theta^{-1}$} (9)
   edge [bend left=7]node[below] {$ \theta^{-1}$} (6);
\end{tikzpicture}
\caption{}
\end{figure}

\end{center}

More simply,
\begin{center}
\begin{tikzpicture}[->,>=stealth',shorten >=1pt,auto,node distance=1.8cm,main node/.style={font=\sffamily\bfseries}]
\node[main node] (1) {$ (C,(x_i)_{i\in S})$};
 \node[main node] (2) [right of=1]{};
 \node[main node] (3) [right of=2] {$(C,(\theta(x_i))_{i\in S})$};
\node[main node] (4)[below of=2] {$ T$};

 \path[every node/.style={font=\sffamily\small}]
   (1) edge node[above] {$ \theta$} (3)
    edge node [left] {$\pi$} (4) 
   (3)edge node[right] {$ \pi\circ \theta^{-1}$} (4);
\end{tikzpicture}
\end{center}

\subsection{Monad structure}
As defined in section 2.4, let us consider the category $\M_{0,S}$ of $S$--labeled stable genus 0 curves, the category $Sch_{\Q}$ of schemes over $\Q$ and the functor $b:\M_{0,S} \to Sch_{\Q}$. We will rely on notions of monads and endofunctors, such as defined in~\cite{MaL98}, Chap. VI.
 
 Consider an involutive endofunctor $\theta:\M_{0,S} \to \M_{0,S} $ i.e. an endofunctor verifying the following condition: \[\theta^2=\theta\circ \theta=Id_{\M_{0,S}}.\] This operation defines a monad in the category $\M_{0,S}$. 

The transformations $\eta:I_{\M_{0,S}}\xlongrightarrow{*}\theta$ and $\mu:\theta^2\xlongrightarrow{*}\theta$ are natural, and the following diagrams commute:
 
 \begin{center}
\begin{tikzpicture}[->,>=stealth',shorten >=1pt,auto,node distance=1.5cm,main node/.style={font=\sffamily\bfseries}]
\node[main node] (1) {$ \theta^3=\theta$};
 \node[main node] (2) [right of=1]{};
 \node[main node] (3) [right of=2] {$ Id_{\M_{0,S}}$};
\node[main node] (4)[below of=1] {$ Id_{\M_{0,S}}$};
\node[main node] (5)[below of=3] {$\theta $};

 \path[every node/.style={font=\sffamily\small}]
   (1) edge node[above] {$\theta_{\mu}$} (3)
    edge node [left] {$\mu_{\theta}$} (4) 
   (3)edge node[right] {$\mu$} (5)
(4)edge node[below] {$ \mu$} (5) ;
\end{tikzpicture}
\end{center}

\vspace{3pt}
\begin{center}
\begin{tikzpicture}[->,>=stealth',shorten >=1pt,auto,node distance=1.5cm,main node/.style={font=\sffamily\bfseries}]
\node[main node] (1) {$ \theta^3=\theta$};
 \node[main node] (2) [right of=1] {};
\node[main node] (4)[right of=2] {$ \theta^2=Id_{\M_{0,S}}$};
\node[main node] (5)[right of=4] {}; 
\node[main node] (7)[right of=5] {$ \theta Id_{\M_{0,S}}$};
 \node[main node] (8) [below of=1] { $ \theta$};
 \node[main node] (9) [below of=4] {$\theta$};
 \node[main node] (10) [below of=7] {$ \theta$};
 
 \path[every node/.style={font=\sffamily\small}]
 (1) edge node[above] {$ \eta\theta$} (4)
  edge node[left] {$ =$} (8)
  (4) edge node[above] {$ \theta\eta$} (7)
  (7)edge node[right] {$ =$} (10)
   (8)edge node[below] {$ =$} (9)
  (9) edge node[below] {$=$} (10);
\end{tikzpicture}.
\end{center}

In particular, we choose the $\theta$-algebra for the monad $\theta$ to be constructed on the model of the group action $G\times\M_{0,S}$, with the structure map $H: G\times \M_{0,S} \to \M_{0,S} $ verifying $H(g_1g_2,x)=H(g_1,H(g_2,x))$, $H(u,x)=x$. This leads to defining the category of $S$--labeled stable curves of genus 0, obtained by the action of $G$ on objects of $\M_{0,S}$, denoted by $\M_{0,S}^{\theta}$. 

In addition, this defines a commutative diagram in the flavour of diagram\, \ref{F:dia}, where we replace $C_{x_1,...,x_n}$ by $\M_{0,S}$, $C_{\theta(x_1),...,\theta(x_n)}$ by $\M_{0,S}^{\theta}$ and $T$ by $Sch_{\Q}$. The dashed arrow corresponds to $b:\M_{0,S} \times \M_{0,S}^{\theta}\to Sch_{\Q}$. This leads to the next proposition. 

\begin{proposition}
Let $\M_{0,S}$ be the category of $S$--labeled stable curves of genus 0 and $\M_{0,S} ^{\theta}$ the category of $S$--labeled stable curves of genus 0, obtained by the action of $G$ on objects of $\M_{0,S} $. 
Then, $b:\M_{0,S} \times \M_{0,S}^{\theta}\to Sch_{\Q}$ is a groupoid. 
\end{proposition}
\begin{proof}
Construct the isomorphism of categories 
${\bf \theta}: \M_{0,S} \to \M_{0,S}^{\theta}$, where any object of $\M_{0,S}^{\theta}$ is the image of one object $C/T$ by the map ${\bf \theta}\in \textrm{Aut}(C/T)$; and any morphism of $S$--labeled stable curves $c,c'\in\M_{0,S}$ is mapped to a morphism of $S$--labeled stable curves in $\M_{0,S}^{\theta}$: ${\bf \theta} f: {\bf \theta}c \to {\bf \theta} c'$.
The definition of this morphism induces a bijection of the two families of $S$--labeled curves in $\M_{0,S}$ and respectively in $\M_{0,S}^{op}$, in other words the following diagram is commutative:
\begin{center}
\begin{tikzpicture}[->,>=stealth',shorten >=1pt,auto,node distance=1.5cm,main node/.style={font=\sffamily\bfseries}]
\node[main node] (1) {$c$};
 \node[main node] (2) [right of=1]{};
 \node[main node] (3) [right of=2] {$\theta c$};
\node[main node] (4)[below of=1] {$ c^{\prime}$};
\node[main node] (5)[below of=3] {$\theta c^{\prime} $};
 
  \path[every node/.style={font=\sffamily\small}]
   (1) edge node[above] {$\theta$} (3)
    edge node [left] {$f$} (4) 
   (3)edge node[right] {$\theta f$} (5)
(4)edge node[below] {$ \theta$} (5) ;
\end{tikzpicture}
\end{center}
Now, we show that $(b,b'):\M_{0,S}\times \M_{0,S}^{\theta}\to Sch_{\Q}$ is a groupoid. Indeed, given a base $T \in Sch_{\Q}$, any morphism of families over $T$ inducing identity on $T$ is an isomorphism for $\M_{0,S}$ and, by the construction above for $\M_{0,S}^{\theta}$.
Hence, the {\it condition for groupoids} in section 2.1 is satisfied, and we have defined a groupoid. 
\end{proof}


\section{Operad for curves with $\theta$-symmetry}\label{S:OpSC}
In this section, we develop an operad for $n$-pointed genus 0 curves, endowed with the transformation $\theta$. Throughout the paper, we call this operad {\it NY operad}. 

The $S$--labeled stable curves, equipped with the bi-functor $``\otimes"$ form a monoidal category. In particular, there exists a well-defined topological operad $\mathcal{P}$ for the $S$--labeled stable curves of genus 0 (see previous section). By symmetry, the $S$--labeled stable $\theta$-curves (i.e. obtained by using the involution $\theta$) form also a monoidal category and lead to a well-defined topological operad $\mathcal{Q}$ for the $S$--labeled stable $\theta$-curves.


Define a morphism of operads in categories $\mathcal{P} \to \mathcal{Q}$. This is a sequence of functors $\mathcal{P}(r) \to \mathcal{Q}(r)$, preserving the internal structures associated to the operads. Therefore, the stable curves, indexed by the graph $({\bf x}, \theta({\bf x}))$ inherit a monoidal structure. 

Let ${\bf C'}$ be the monoidal category of curves with $\theta$-symmetry (i.e. even pointed curves), indexed by the graph $({\bf x}, \theta({\bf x}))$; morphisms are bijections (see section\, \ref{s:sdm}). We show that there exists a topological operad on the monoidal category ${\bf C'}$. To construct this operad we use a monoidal functor between the classical monoidal category ${\bf C}$ of $S$--labeled stable curves and the monoidal category ${\bf C'}$ of curves with the $\theta$-symmetry.

For the sake of clarity let us recall a few notions and notations concerning monoidal functors and Lax monoidal functors.
\begin{definition}[Monoidal functor]\label{D:MonoF}
 A monoidal functor $\Phi=(F_{1},F_2,F_0):\Cb \to \Cb'$ between monoidal categories $\Cb$ and $\Cb'$ consists of the following items:
 \begin{itemize}
 \item An ordinary functor $F_1: \Cb \to \Cb'$ between categories;
 \item For objects $a,b$ in $\Cb$ morphisms: 
 \begin{equation} 
 F_2(a,b):F(a)\otimes F(b)\to F(a\otimes b)
\end{equation} in $\Cb'$ which are natural in $a$ and $b$
 \item For the units $e$ and $e'$, a morphism in $\Cb'$
 \begin{equation}
 F_0:e'\to Fe \end{equation}
 \end{itemize}
 Together these must make all the following three diagrams involving the structural maps $\alpha, \lambda$ and $\rho$ commute in $\Cb'$.

 \begin{center}
\begin{tikzpicture}[->,>=stealth',shorten >=1pt,auto,node distance=1.7cm,main node/.style={font=\sffamily\bfseries}]
\node[main node] (1) {$\scriptstyle F(a)\otimes (F(b)\otimes F(c))$};
 \node[main node] (2) [right of=1] {};
 \node[main node] (3) [right of=2] {};
 \node[main node] (4) [right of=3] {$\scriptstyle (F(a)\otimes F(b))\otimes F(c)$};
\node[main node] (5)[below of=1] {$\scriptstyle F(a)\otimes (F(b\otimes c))$};
 \node[main node] (6) [right of=5] {};
 \node[main node] (7) [right of=6] {};
 \node[main node] (8) [right of=7]{$\scriptstyle F(a \otimes b)\otimes F(c))$};
 \node[main node] (9) [below of=5] {$\scriptstyle F(a\otimes (b\otimes c))$};
  \node[main node] (10) [right of=9] {};
  \node[main node] (11) [right of=10] {};
 \node[main node] (12) [right of = 11] {$\scriptstyle F(a \otimes b)\otimes c)$};

 \path[every node/.style={font=\sffamily\small}]
   (1) edge node[above] {$\scriptstyle \alpha$} (4)
      edge node [left] {$\scriptstyle 1\otimes F_{2}$} (5) 
    (4)edge node[right] {$\scriptstyle F_{2}\otimes 1$} (8)
   		(5) edge node [left] {$\scriptstyle F_{2}$} (9)
   		  (8) edge node[right] {$ \scriptstyle F_{2}$} (12)
    (9)edge node [below] {$\scriptstyle F_{1}(\alpha)$} (12);
\end{tikzpicture}\label{E:12}
\end{center}

\begin{center}
\begin{tikzpicture}[->,>=stealth',shorten >=1pt,auto,node distance=2.5cm,main node/.style={font=\sffamily\bfseries}]
\node[main node] (1) {$\scriptstyle F(b)\otimes e'$};
 \node[main node] (2) [right of=1] {};
\node[main node] (3)[right of=2] {$\scriptstyle F(b)$};
 \node[main node] (4)[below of = 1]{$\scriptstyle F(b)\otimes F(e)$};
 \node[main node] (5) [right of=4] {};
 \node[main node] (6) [below of = 3] {$\scriptstyle F(b\otimes e)$};

 \path[every node/.style={font=\sffamily\small}]
   (1) edge node[above] {$\scriptstyle \rho$} (3)
    edge node [left] {$\scriptstyle 1\otimes F_{0}$} (4) 
   (4) edge node [below] {$\scriptstyle F_{2}$} (6)
    (6)edge node[right] {$\scriptstyle F(\rho)$} (3);   
\end{tikzpicture}\label{E:2}
\end{center}

\begin{center}
\begin{tikzpicture}[->,>=stealth',shorten >=1pt,auto,node distance=2.5cm,main node/.style={font=\sffamily\bfseries}]
\node[main node] (1) {$\scriptstyle e'\otimes F(b)$};
 \node[main node] (2) [right of=1] {};
\node[main node] (3)[right of=2] {$\scriptstyle F(b)$};
 \node[main node] (4)[below of = 1]{$\scriptstyle F(e)\otimes F(b)$};
 \node[main node] (5) [right of=4] {};
 \node[main node] (6) [below of = 3] {$\scriptstyle F(e\otimes b)$};

 \path[every node/.style={font=\sffamily\small}]
   (1) edge node[above] {$\scriptstyle \lambda$} (3)
    edge node [left] {$\scriptstyle F_{0}\otimes 1$} (4) 
   (4) edge node [below] {$\scriptstyle F_{2}$} (6)
    (6)edge node[right] {$\scriptstyle F(\lambda)$} (3);   
\end{tikzpicture}\label{E:3}
\end{center}
\end{definition}

\begin{proposition}\label{P:1}
Lax monoidal functors send monoids to monoids: 
if $F:(\Cb,\otimes)\to (\Cb',\otimes)$ is a lax monoidal functor and 
\begin{equation}
A\in \Cb, \mu_A:A\otimes A \to A, i_A:I\to A
\end{equation}
is a monoid object in $\Cb$, the object $F(A)$ is naturally equipped with the structure of a monoid in $\Cb'$ by setting 
\begin{equation}
i_{F(A)}:I_{\Cb'}\to F(I_{\Cb})\xrightarrow{F(i_{A})} F(A)
\end{equation}
and
\begin{equation}\label{E:Mon}
\mu_{F(A)}:F(A)\otimes F(A)\to F(A\otimes A)\xrightarrow{F(\mu_{A})} F(A). 
\end{equation}
This construction defines functor,
\[Mon(f):Mon(\Cb)\to Mon(\Cb').\]
\end{proposition}

\begin{lemma}
Let ${\bf C'}$ be the monoidal category of curves with $\theta$-symmetry defined above. 
Then, there exists an operadic structure on the monoidal category ${\bf C'}$.
\end{lemma}
\begin{proof}
Let $\Phi: {\bf C}\to {\bf C'}$ be a monoidal functor between the categories ${\bf C}$ and ${\bf C'}$ defined above. 

First consider the category ${\bf C}$. By applying definition \ref{D:Op}, we have a presheaf: $\Ps:Fin^{op}\to {\bf C}$ that defines an operad structure. 

Using the functor defined in equation~\ref{E:Mon}, Prop.~\ref{P:1}, we define $Mon(f):Mon({\bf C})\to Mon({\bf C'})$, mapping the monoid in ${\bf C}$ to the monoid in ${\bf C'}$. Using the relations in the diagrams of definition \ref{D:MonoF}
we have that 
 \[\Phi(\Ps):Fin^{op}\to {\bf C'}\] is a presheaf.
In fact, from proposition~\ref{P:1}, it follows that the composition operation is conserved in {\bf C'}, i.e. one has: 
\[\Phi(\Ps(Y\setminus X)\otimes \Ps(X))=\Phi(\Ps(Y\setminus X)) \otimes \Phi(\Ps(X)).\]
Therefore, the morphism $\Phi(\Ps(Y\setminus X)\otimes \Ps(X))\to \Phi(\Ps(Y))$ turns out to be:
\[ \Phi(\Ps(Y\setminus X)) \otimes \Phi(\Ps(X))\to \Phi(\Ps(Y)),\]
which defines what we wanted.
\end{proof}

Therefore, we may formulate the following definition of the topological $NY$ operad.


\begin{definition}[NY Operad]

Let {\bf C} and ${\bf C'}$ be monoidal categories of labeled curves defined above and $\Phi: {\bf C } \to {\bf C'}$ the monoidal functor. Let $\Ps$ be the presheaf giving the operad structure on $\{\M_{0,n+1}\}_{n\geq3}$. We define the $NY=\{NY(n+1)\}_{n\geq 3}$ to be the operad in the category ${\bf C'}$, given by the presheaf 
$\mathcal{P'}=\Phi\circ \Ps=: Fin^{op}\to {\bf C'}$ and endowed with the partial operadic composition $\circ_{X\subset Y}$:
\[\mathcal{P'}(Y/X)\otimes \mathcal{P'}(X)\to \mathcal{P}(Y),\] for any $X\subset Y$ and a unit $\eta: 1_{{\bf C}}\to \mathcal{P'}(\{*\}).$
\end{definition}

An object $NY(n+1)$ of ${\bf C'}$ can be described more geometrically. It is a sphere (i.e. $\PC$) with $2n-3$ marked points differing from $\{0,\frac{1}{2}, 1, \infty\}$. On this sphere, there exists one point (at $Fix(\theta)$) which is nodal. Therefore, we blow-up this point, and obtain two $\PC^{1}$'s transversally intersecting at a point; one of which has $2n-4$ marked points plus the three points: $\{0,1,\infty\}$, and the other one has 2 marked points. The operadic composition, done for $\M_{0,n+1}$, is symmetrically perpetuated on those marked points which are obtained by $\theta$, by the morphism argument mentioned above. 

\begin{remark}
The notation $\mathcal{M}_{0,2n+3}$, where $n\geq 1$, stands in this remark for the smooth stratum of the compactified space $\M_{0,2n+3}$. The smooth stratum equipped with $\theta$-symmetry
\[\{(z_1,\dots, z_n, \theta(z_1),\dots, \theta(z_n),0,1,\infty)\in \PC^{2n+3} \, | \, z_i\in \PC^{1} \setminus \{0,\frac{1}{2},1,\infty\}, z_i\neq z_j\}\] is a subspace of $\mathcal{M}_{0,2n+3}$.
It is constructed by adding the three points $(0,1,\infty)$ to the linear space as follows:

\[((\PC \setminus \{0,1,\infty\})^{n} \setminus \bigcup_{i=1}^{n} \{x_i=\frac{1}{2} \})\, \times\, ((\PC \setminus \{0,1,\infty\})^{n} \setminus \bigcup_{i=1}^{n} \{x_i=\frac{1}{2}\}). \]

Naturally, the composition operation for the NY operad is inherited from the one of $\{\M_{0,2n+3}\}_{n\geq 0}$:
\[\M_{0,2n+3}\circ_i \M_{0,2m+3}\to \M_{0,2(n+m+2)+3}.\]
\end{remark}
This remark is illustrated by the following investigation. 

$\ast$ The codimension 0 stratum of $NY$ is obtained by cutting out the fixed point under the transformation $\theta$, which forms a singularity. The smooth stratum is given by the set of $n$-tuples:

 \[\{(x_1,x_2,...,x_{2n},0,1,\infty)\in \PC^{2n+3}\, |\, x_i \in \PC \setminus \{0,\frac{1}{2},1,\infty\}, x_i\neq x_j \}.\] 
 The codimension 0 stratum of $NY(n)$ is contained in $\mathcal{M}_{0,2n+3}$, defined as:
 \[\mathcal{M}_{0,2n+3}=\{(x_1,x_2,...,x_{2n}, 0,1,\infty)\in \PC^{2n+3} \ | \ x_i \in \PC \setminus \{0,1,\infty \}, x_i\neq x_j \}.\]

$\ast$ The codimension 1 stratum: \[(\mathcal{M}_{0,2n+3}\setminus \, \bigcup_{i=1}^{n}\{x_i=\theta(x_i)=\frac{1}{2}\}) \times (\mathcal{M}_{0,2m+3}\setminus \, \bigcup_{i=1}^{m}\{x_i=\theta(x_i)=\frac{1}{2}\}), \] is isomorphic to a subspace of the codimension 1 stratum $\mathcal{M}_{0,2n+3} \times \mathcal{M}_{0,2m+3}$.

\vskip.1cm 
The composition map is illustrated in the following way.
\begin{example} In the simplest case i.e. when $n=1$, we have the moduli space given by: \[\{(x_1,\theta(x_1), 0,1,\infty)\in \PC^{5} \ | \ x_i\in \PC^{1}- \{0,\frac{1}{2},1,\infty)\},\] where the transformation $\theta$, implies $\theta(x_1)=1-x_1\in \PC^{1} \setminus \{0,\frac{1}{2},1,\infty\}.$ This is a sub-moduli space of $\mathcal{M}_{0,5}$. We now endow this space with a composition map $``\circ"$ to form an algebra structure (and thus the NY operad).

\vskip.1cm 

 Let us compose this space, for instance, with the moduli space: \[\{(y_1,y_2, \theta(y_1),\theta(y_2),0,1,\infty)\in \PC^{7} \ |\ y_i\in \PC^{1} \setminus \{0,\frac{1}{2},1,\infty\} , y_1\neq y_2\}.\]  
The composition is given by using the subspace of the codimension 1 stratum $\mathcal{M}_{0,5}\times \mathcal{M}_{0,7}$ and results in the following space: 
\[\{(z_1,z_2,z_3, z_4, \theta(z_1),\theta(z_2),\theta(z_3),\theta(z_4),0,1,\infty)\in \PC^{11} \ |\ \ z_i\in \PC^{1} \setminus \{0,\frac{1}{2},1,\infty\} , z_i\neq z_j\}.\]
\end{example}  
The generalisation to other subspaces or higher arities can be done following the same recipe in terms of codimension 1 strata.
 

\section{The DM--stack with hidden symmetry}\label{S:NY}

It has been previously shown that the monoidal category of DM--stacks, enriched by the $\theta$-symmetry, gives a collection of objects, forming an operad denoted $NY=\{NY(n+1)\}_{n\geq1}$. We use this $NY$ operad to construct a new type of gravity operad, in the flavour of~\cite{Get},~\cite{Kap}.

\subsection{Stratification of $\M_{0,n}$ with $\theta$-symmetry}\,

We show properties of the stratification of $\M_{0,n}$ with symmetry. Consider an object $(C/T,(x_i)_{i\in S})$, where $C$ belongs to the category $\mathcal{F}$ discussed in section~\ref{s:F}, and $T$ belongs to the category $\mathcal{S}$. The scheme $T\in \mathcal{S}$ is decomposed into a disjoint union of strata, where each stratum $D(\tau)$ is indexed by a stable $S$-tree $\tau$. The stratum $D(\tau)$ is a locally closed, reduced and irreducible subscheme of $T$, and the parametrizing curves are of the combinatorial type $\tau$. The codimension of the stratum $D(\tau)$ equals to the cardinality of the set of edges of $\tau$ (i.e. the number of singular points of a curve of type $\tau$). Note that this subscheme depends only on the $n$-isomorphism classes of $\tau$. 

The graph, $\tau$, is the dual graph. It is obtained by blowing up the colliding points in the fat diagonal $\Delta$ (where $\Delta$ is given by $\{x_i=x_j | i\neq j\}$). The fat diagonal forms an $\mathcal{A}_{n-1}$ singularity. By Hironaka~\cite{Hi}, one can blow-up the singular locus in such a way that it gives a divisor with normal crossings. This divisor is usually called $D$ and lies in the total space. To avoid any type of confusion, we denote by $D^{\tau}$ a divisor with normal crossings, lying in the total space $C$ and with combinatorial data encoded by the graph $\tau$.

The divisor $D^\tau$ corresponds to the given stratum $D(\tau)$ of $\M_{0,n}$ (the base space). It is decomposed into a sum of irreducible components which are closed integral subschemes of codimension 1, on the blown-up algebraic variety. We have\[ D^{\tau}=\sum_{n\in I} n_iD_{\tau, i},\] where the $n_{i}$ are integers.
Each of the irreducible components are isomorphic to $\PC^1$. The normal crossing condition implies that each irreducible component is non singular and, whenever $r$ irreducible components $D_{\tau,1},\dots, D_{\tau,r}$ meet at a point $P$, the local equations $f_{1},\dots,f_{r}$ of the $D_{\tau, i}$ form part of a regular system of parameters at $P$. It is also possible to define locally the divisor by $\{(U_i, f_i)\}$, where $f_i$ are holomorphic functions and $U_i$ are open subsets.

The closure of a stratum $\overline{D}(\tau)$ in $T$ is formed from the union of subschemes $D(\sigma)$, where $\tau>\sigma$, such that $\tau$ and $\sigma$ have the same set of tails. The tails correspond to the set of marked points on the irreducible connected component. In the geometric realization of the graph, the tails---also called {\it flags}---correspond to a set of numbered half-edges, which are incident to only one vertex. This set is denoted by $F_{\tau}(v)$, where $v$ is a given vertex in the graph $\tau$, to which the half-edges are incident. In our case, where the genus of the curve is zero, the condition $\tau>\sigma$ is uniquely specified by the {\it splitting} data, which can be described as a certain type of Whitehead move. 

Let us expose roughly the splitting procedure. Choose a vertex $v$ of $\tau$ and a partition of the set of flags, incident to $v$: $F_{\tau}(v)=F'_{\tau}(v)\cup F''_{\tau}(v)$ such that both subsets are invariant under the involution $j_{\tau}: F_{\tau}\to F_{\tau}$. To obtain $\sigma$, replace the vertex $v$ by two vertices $v'$ and $v''$ connected by an edge $e$, where the flags verify $F'_{\tau}(v')=F'_{\tau}(v)\cup \{e'\}$, $F'_{\tau}(v'')=F''_{\tau}(v)\cup \{e''\}$, where $e',e''$ are the two halves of the edge $e$. The remaining vertices, flags and incidence relations stay the same for $\tau$ and $\sigma$. For more details, see~\cite{Ma99} Chap. III $\S$ 2.7, p.90.

Consider the stratification of the scheme $T$ by graphs (trees, in fact), such as depicted in~\cite{Ma99} in Chap. III $\S 3$. Let $\tau$ be a tree. If $S$ is a finite set, then $\mathcal{T}((S))$ is the set of isomorphism classes of trees $\tau$, whose external edges are labeled by the elements of $S$. The set of trees is graded by the number of edges:
\[\mathcal{T}((S))=\bigcup_{i=0}^{|S|-3}\mathcal{T}_{i}((S)), \]
where $\mathcal{T}_{i}((S))$ is a tree with $i$ edges. The tree $\mathcal{T}_{0}((S))$ is the tree with one vertex and the set of flags equals to $S$. 

Not only do those trees allow a stratification of $\M_{0,n}$, but due to~\cite{{Kee},{EHKR},{Sing}}, one may derive relations for computing the cohomology ring $H^*(\M_{0,n}(\C))$. In particular, $\M_{0,n}$ has no odd homology and its Chow groups are finitely generated and free abelian.

We recall briefly the Keel presentation~\cite{Kee}. Let $D_S$ be a component of a divisor lying in $\M_{0,n}$, where $S\subset \{1,2,...,n\}$. In Keel's presentation (see~\cite{Kee} section 1), the $n$-th Chow ring turns out to be isomorphic to the quotient of the polynomials ring, generated by degree 2 elements (the $D_S$), modulo an ideal in which the $D_S$ are subject to the following relations:

\begin{itemize}
\item $D_S=D_{\{0,1,...,n\}\setminus S}$.
\item For distinct elements $i,j,k,l\in\{0,1,...,n\}$ \[\sum_{i,j\in S; k,l\notin S }D_S=\sum_{i,k\in S; j,l\notin S } D_S.\]
\item If $S\cap T \notin \{0,S,T\}$ and $S\cup T \neq \{0,1,...,n\}$ then $D_SD_T=0$.
\end{itemize}
In~\cite{Do}, it is shown that $H^{*}(\M_{0,n})$ is Koszul, which is useful to determine various homotopy invariants for the DM--compactification.
\smallskip

We proceed to a brief discussion concerning the stratification of the $S$-stable curves of genus 0 with $\theta$-symmetry $\M_{0,S}(\theta)$, i.e. indexed by the graph $(x,\theta(x))$ and compare it to the strata in the classical $\M_{0,S}$ case. A list of the first strata occurring in the base space for $\M_{0,S}(\theta)$ is given. The word ``codim'' stands for codimension:
\begin{center}
\begin{table}[h]
\begin{tabular}{clll}
 Codim $n$ & & & \\
$n=0$ & $(x_1,..,x_n)$:&$ x_i \neq x_j$ for $i\neq j$, &$x_i \in \PC^1\setminus \{0,1,\frac{1}{2},\infty \}$.\\
$n=1$ & $(x_1,..,x_n)$:& $x_i \neq x_j$ for $i\neq j$ , & $x_i \in \PC^{1}\setminus \{0,1,\infty \}$.\\
$n=1$ & $(x_1,..,x_n)$:&$ x_{i_1}=x_{i_2}$,& $x_i \in \PC^{1}\setminus \{0,1,\frac{1}{2},\infty\}$,\\
&&$ x_{i_k} \neq x_{i_j}$  for $k,j \not\in \{1,2\}$,&$ x_i \in \PC^{1}\setminus \{0,1,\frac{1}{2},\infty\}$.\\
$n=2$ & $(x_1,..,x_n)$ :& $x_{i_1}=x_{i_2}=x_{i_3}$, & $x_i \in \PC^{1}\setminus \{0,1,\infty \}$.\\
$n=2$ & $(x_1,..,x_n)$:&$x_{i_2}=x_{i_3}$, & $x_i \in \PC^{1}\setminus \{0,1,\infty \}$,\\
&&& $x_{i_1}{\small\in}\PC^{1}\setminus \{0,1,\frac{1}{2},\infty \}$.\\
$n=3$ & $(x_1,..,x_n) $:&$ x_{i_1}=x_{i_2}=x_{i_3}=x_{i_4}$, & $x_i \in \PC^{1}\setminus \{0,1,\infty \}$. \\
$n=3$ & $(x_1,..,x_n)$:&$ x_{i_2}=x_{i_3}=x_{i_4}$, & $x_i \in \PC^{1}\setminus \{0,1,\infty \}$,\\
&&&$x_{i_1}{\small\in}\PC^{1}\setminus \{0,1,\frac{1}{2},\infty \}$.\\
$n=4$ & $(x_1,..,x_n) $:&$ x_{i_1}=x_{i_2}$, &$x_i \in \PC^{1}\setminus \{0,1,\frac{1}{2},\infty \}$.\\
$n=4$ & $(x_1,..,x_n)$:&$ x_{i_1}=...=x_{i_5}$, & $x_i \in \PC^{1}\setminus \{0,1,\infty \}$. \\
\end{tabular}
\end{table}
\end{center}

This discussion, allows further considerations concerning the fibers of the proper flat map $\pi:C\to T$, where $C$ is the $|S|$-stable curve. To a given divisor $D^{\tau}=\sum_{n\in I} n_iD^{\tau}_{i}$, locally defined by a chart $\{(U_i, f_i)\}$, the symmetry $\theta$ maps $D^{\tau}$ to the corresponding divisor: $D^{\theta(\tau)}=\sum_{n\in I} n_iD^{\theta(\theta)}_{i}$ defined by $\{(U_i, \theta(f_i))\}$. \
So, this defines a pair of divisors $(D^\tau, D^{\theta(\tau)})$, being isomorphic. The intersection $(D_\tau,D^{\theta}_\tau)$ is non-empty if and only if $f_i=\theta(f_i)$. 

\vskip.1cm

\begin{proposition}
Any pair of divisors $D^\tau$ and $D^{\theta(\tau)}$ indexing a given stratum in $\M_{0,S}$ and $\M_{0,S}(\theta)$ are a point reflection of each other and their union in $\M_{0,S} \times \M_{0,S}(\theta)$ 
forms a connected set if and only if there exists a holomorphic function such that $f_i=\theta(f_i)$. 
\end{proposition} 
\begin{proof}
One direction is easy: if $D^{\theta(\tau)}$ and $D^{\tau}$ are disjoint, then there are no points verifying $f_i=\theta(f_i)$. This occurs in the case where the sections are different from $\frac{1}{2}$, so lying in $\PC^{1}\setminus\{0,1,\frac{1}{2},\infty\}$. Reciprocally, if there exists a chart for which $f_i=\theta(f_i)$, then the set $D^\tau \cup D^{\theta(\tau)}$ is connected. It is clear that $D^{\theta(\tau)}$ is a point reflection of $D^\tau$: each component of the divisor $D^\tau,$ being given by the charts an $\{(U_i, f_i)\}$, where $U_i$ is in $\C$ and $f_i$ is a holomorphic function, therefore the charts for $D^{\theta(\tau)}$ given by $\{\theta(U_i), \theta(f_i)\}$ are given by the point the point reflection $\theta$, so concerning the geometry of the divisor $D^{\theta(\tau)}$ it is a point reflection of $D^{\tau}$. 
\end{proof}

\begin{proposition}
The graphs, indexing the strata of $\M_{0,n}\times \textrm{Aut}(\M_{0,n})$, are invariant under a symmetry group of order 2. 
\end{proposition}
\begin{proof}
The hyperplane arrangements giving the set of colliding points, for $\M_{0,n}$, resp. $ \textrm{Aut}(\M_{0,n})$, are symmetric to each other by the point reflection $\theta$.
Therefore, since a divisor is defined in exactly the same way, i.e. in local coordinates we have $f_1...f_i=0$, (resp. $\theta(f_1)...\theta(f_i)=0)$, the divisors for $\M_{0,n}$ and $ \textrm{Aut}(\M_{0,n})$ are isomorphic to each other, and glued at a point. This implies that the dual graphs of $\M_{0,n}\times \textrm{Aut}(\M_{0,n})$, are invariant under a symmetry group of order 2.
\end{proof}

\subsection{The stacky $\M_{0,S}(\theta)$}

\begin{theorem}
The category fibered in groupoid $\M_{0,S}(\theta)$ over $Sch_{\Q}$ of $S$-pointed stable curves with $\theta$-symmetry is a stack.
\end{theorem}
\begin{proof}
In order to show that $\M_{0,S}(\theta)$ is a stack, let us first equip the base space $\cS$ with the \'etale topology $\mathcal{T}$. We need to verify the three conditions of definition~\ref{DM}.
\begin{enumerate}
\item The first condition is to show that the contravariant functor from $\cS^{op}$ to the category of sets $Set$ is a sheaf. 

We know that if $\cS$ has a Grothendieck topology, then $\M_{0,n}$ is a stack. The modification of this stack into $\M_{0,(x,\theta(x))}$ implies a slight modification of the data. Indeed, $\F: \cS^{op}\to Set$ is a sheaf. Properties of the category $Set$ allow to consider the direct sum $Set\oplus Set$.
So, we are dealing with the section $ \cS^{op}\to Set\oplus Set$, which turns out to be a direct sum of sheafs: $\F + \F: \cS^{op}\to Set\oplus Set$, hence a sheaf. So, we have a sheaf $\cS^{op}\to Set$.  

\item The second condition to check is that for any open $T'$ in $T$, the functor $T'\mapsto Iso_{T'}(X_1\oplus X_1^{\theta},X_2\oplus X_2^{\theta})$ is a sheaf.

By hypothesis, we know that for any open $T'$, the functor $T'\mapsto Iso_{T'}(X_1,X_2)$ is a sheaf. Clearly, the functor $T'\mapsto Iso_{T'}(X_1^{\theta},X_2^{\theta})$ is also a sheaf. 
So, the map from $T'\mapsto Iso_{T'}(X_1,X_2)\oplus Iso_{T'}(X_1^{\theta},X_2^{\theta})$ is a sheaf and by elementary properties of $\oplus$, we have $Iso_{T'}(X_1,X_2)\oplus Iso_{T'}(X_1^{\theta},X_2^{\theta})=Iso_{T'}(X_1\oplus X_1^{\theta},X_2\oplus X_2^{\theta})$. 
 
\item The last property is the so-called cocycle condition. Let $\{T_i \xrightarrow{\phi_i} T\}_{i}$ be an \'etale cover of $T$, where $\phi_i$ are \'etale maps and let $F$ be a family over $T$. Then, applying to $F$ the base change functors $\phi^*$, we get localized families $F_i$ over $T_i$, and similarly localized families $F_{ij}$ over $T_{ij}:=T_i\times_{T}T_j$, $F_{ijk}$ over $T_{ijk}$, (etc).

 They come along with the descent data, i.e isomorphisms $f_{ij}:pr^*_{ji,i}F_i\xrightarrow{\cong} pr^*_{ji,j}F_j$, which turn to satisfy the cocycle condition: $f_{ki}=f_{kj}\circ f_{ji}$ on $T_{kji}$. The family $F$ is compatible with the direct sum operation. 
 
Therefore, we have over $T$ the family: $F\oplus F^{\theta}=(C/T,(x_i))\oplus (C/T,(\theta(x_i)))$. The base change functors $\phi^*$ give localized families $F_i\oplus F^{\theta}_{i}$ over $T_i$ (more generally, $F_{ij}\oplus F^{\theta}_{ij}$ over $T_{ij}:=T_i\times_{T}T_j$, $F_{ijk}\oplus F^{\theta}_{ijk}$ over $T_{ijk}$, etc).

We have $(f_{ij},f^{\theta}_{ij}):(pr^*_{ji,i}F_i,pr^*_{ji,i}F^{\theta}_i) $$\xrightarrow{(\cong,\cong)}(pr^*_{ji,j}F_j,pr^*_{ji,j}F^{\theta}_{}j)$, satisfying 
the cocycle condition: $(f_{ki},f^{\theta}_{ki})=(f_{kj}\circ f_{ji}, f^{\theta}_{kj}\circ f^{\theta}_{ji})$.
The converse property comes from 2). 
 \end{enumerate}
 \end{proof}
The construction of $\M_{0,S}(\theta)$ being very close and in some sense inherited from $\M_{0,S}$, we can state the following. 
 \begin{corollary}
The category fibered in groupoid $\M_{0,S}(\theta)$ over $Sch_{\Q}$ of $S$-pointed stable curves with $\theta$-symmetry is an algebraic stack of Deligne--Mumford type.
 \end{corollary}
 
\subsection{The NY Gravity operad}\,
In this section, we introduce the NY gravity operad $Grav_{NY}$.
The terminology of {\it gravity} was chosen here because of the geometric structure of the space $\M_{0,n}$ with $\theta$-symmetry. In particular, this is equivalent to introducing a split quaternionic structure, given by $\theta$. Such a structure can interpreted as the realisation of a module over a ring extending the complex numbers: the ring of split quaternion numbers. This type of model serves in gravity and general relativity, as one can see for instance in~\cite{KuKoGev} and~\cite{Ul}.


Actually, whenever we want to take into account action of symmetries
upon complexes (de Rham, Hodge etc) and respective mixed structures,
we are bound to pass to one of the several derived environments.
For example, one can use a quite general ``dendroidal'' formalism of \cite{CiMoe},
perhaps extended by various versions of graph categories developed
in \cite{BoMa}, or even Feynman categories.

\smallskip

For the NY gravity operad, we take inspiration in a similar construction as in~\cite{Get}. In the latter case, the gravity operad $Grav=\{Grav(n)\}_{n\geq 0}$ with $n$-arity $Grav(n)=sH_{*}(\mathcal{M}_{0,n})$ where $s$ denotes the suspension, has a mixed Hodge structure.

$Grav$ lies in the symmetric monoidal category of mixed Hodge complexes. The monoidal structure is given by the graded tensor product. Every object $Grav(n)=sH_{*}(\mathcal{M}_{0,n})$ in this category, carries a unique mixed Hodge structure, compatible with the Poincar\'e residue map.

The reason mixed Hodge structures enter the game follows from Proposition 8.2.2 of Deligne in~\cite{De}. This proposition states that for any integer $k$, the $k^{th}$ cohomology group of a given complex algebraic variety is endowed with a functorial mixed Hodge structure. We will investigate similar structures for our NY gravity operad.


By $Fix_{\theta}$ we denote the set of fixed points under the involution $\theta$. The locus, in this context is a point. Since this gives a singular point (or an irreducible curve $\mathbb{P}^1$, if we have proceeded to a blow-up), we remove it, in order to define the NY Gravity operad.
\begin{definition}
Let $Grav_{NY}(n)$ be the stable cyclic $\Ss$-module defined as follows:

\[Grav_{NY}(n)= 
\begin{cases}
sH_*(\mathcal{M}_{0,n}(\theta)\setminus Fix_{\theta}), \quad n\geq 3\\
0, \quad \quad\quad\quad\quad\quad\quad\quad\quad\quad\quad\quad\quad\quad\quad n<3
\end{cases}
\]
\end{definition}
Due to the split quaternionic structure of $\theta$, we have the following implication.
 
\begin{lemma}[Splitting lemma]\label{L:split}
 moduli space with $\theta$-symmetry $\mathcal{M}_{0,n}(\theta)$ splits into a pair of isomorphic and symmetric subspaces $\mathcal{M}_{0,n}(\theta)\cong \mathcal{M}_{0,n}\times \mathrm{Aut}(\mathcal{M}_{0,n})$, where $\mathrm{Aut}(\mathcal{M}_{0,n})$ is the image of $\mathcal{M}_{0,n}$ under $\theta$.
\end{lemma}

We briefly discuss the logarithmic forms along a normal crossing divisor and recall notations from previous sections, namely about the coordinates of the normal crossing divisor. We define a divisor $D$ as $\{f_1\dots f_r=0\}$ in ${\M}_{0,n}$, for a fixed positive integer $r$. A meromorphic differential form on ${\M}_{0,n}$ has a logarithmic form on the divisor if it can be written as a linear combination of forms of the type: 
 \[\frac{\partial f_{i_1}}{f_{i_1}} \wedge \dots \wedge \frac{\partial{ f_{i_r}}}{{f_{i_r}}}\wedge \eta,\]
with $i_1\leq ...\leq i_r$ and where $\eta$ is a holomorphic form on ${\M}_{0,n}$. On each closure of a stratum $\overline{D(\sigma)}$, define a complex of sheaves of logarithmic forms : 
$\Omega_{\overline{D(\sigma)}}^{\bullet}(log {\overline{D(\sigma)}})$. If $j_{\sigma}$:$ D(\sigma)\to \overline{D(\sigma)}$ denotes the natural open immersion, we have a quasi-isomorphism: 
\[ (j_{\sigma})_{\star}(\C_{D(\sigma)})\cong\Omega_{\overline{D(\sigma)}}^{\bullet}(log {\overline{D(\sigma)}}). \]
This induces an isomorphism of cohomology groups. 

\begin{theorem}
 The stable cyclic $\Ss$-module $\{Grav_{NY}(n)\}_{n\geq 3}$ admits a natural structure of cyclic operad and is denoted by $Grav_{NY}$.
\end{theorem}
\begin{proof}
It is known that the collection of stable curves with labeled points forms a topological operad. As shown in equation~\eqref{E:1},
we have a well defined operation composition for this:\[\M_{0,l}\times\M_{0,m_1}\times\dots\times\M_{0,m_l}\to \M_{0,m_1+\dots +m_l}.\] 
By $\theta$-symmetry, the same holds for the collection $\{\mathrm{Aut}(\M_{0,n})\}_{n\geq 3}$.
Combining both morphisms together, we get the following:
\[\begin{aligned}
\M_{0,l}\times\mathrm{Aut}(\M_{0,l}) \times \M_{0,m_1}\times \mathrm{Aut}(\M_{0,m_1})\times & \dots\times \M_{0,m_l}\times \mathrm{Aut}(\M_{0,m_l})\\
&\to \M_{0,m_1+\dots +m_l} \times \mathrm{Aut}(\M_{0,m_1+\dots +m_l}).
\end{aligned}\]
 
To construct the NY gravity operad, we restrict our considerations to the smooth stratum and cut out the fixed point $Fix_\theta$.
 
To define the composition map for $Grav_{NY}$,we proceed by using the Poincar\'e residue morphism~\cite{De}. We need the following embedding : 
\[\begin{aligned}
(\mathcal{M}_{0,l}\times\mathrm{Aut}(\mathcal{M}_{0,l}))\setminus Fix_{\theta} \times (\mathcal{M}_{0,m}\times& \mathrm{Aut}(\mathcal{M}_{0,m}))\setminus Fix_{\theta} \\
& \to \M_{0,m+l}\times \mathrm{Aut}(\M_{0,m+l}),
\end{aligned}\] 

The Poincar\'e residue associated to that embedding is given by: 
\[
Res: H^{*}(\mathcal{M}_{0,m+l}\times \mathrm{Aut}(\mathcal{M}_{0,m+l}))\to H^{*}(\mathcal{M}_{0,m}\times \mathrm{Aut}(\mathcal{M}_{0,m}) \times\mathcal{M}_{0,l}\times \mathrm{Aut}(\mathcal{M}_{0,l})).\]

The composition operation is now discussed. With the residue morphism, defined in the paragraph above, for $\tau>\sigma$ we have the residue formula:
\begin{equation}\label{E:3}
Res_{\tau}^{\sigma}: H^{\bullet}(D(\sigma)\sqcup D(\sigma^{\theta}))\to H^{\bullet-1}(D(\tau)\sqcup D(\tau^{\theta}))(-1), \end{equation}
where we have added a right Tate twist $(-1)$ and $D(\sigma^{\theta})$ (resp. $D(\tau^{\theta})$) is the stratum in $\mathrm{Aut}(\M_{0,m+l})$. 

Let $\sigma\sqcup \sigma^{\theta}\in\mathcal{T}((n))\oplus \mathcal{T}((n))$ in the moduli space $\M_{0,m+l}\times \mathrm{Aut}(\M_{0,m+l})\setminus Fix_\theta$. We choose $\sigma$ (resp. $\sigma^{\theta}$) to have one internal edge $e$ (resp. $e_{\theta}$) joining two internal vertices $v'$ and $v''$. The stratum in $\M_{0,m+l}$, $D(\sigma)$, is decomposed as follows:
\[D(\sigma)\cong \mathcal{M}_{0,F(v')}\times \mathcal{M}_{0,F(v'')},\]
where in the flag $F(v)$ (resp. $F(v'')$) there are $m$ (resp. $l$) incident edges to $v'$ (resp. $v''$). An analogical decomposition is done for $D(\sigma^{\theta})$, where the pair of vertices incident $e_{\theta}$ are denoted $v'_{\theta}$ and $v''_{\theta}$. 

Given $\sigma$ with one edge, we get the residue morphism: 
\[H^{a+b-1}(\mathcal{M}_{0,m+l}\times \mathrm{Aut}(\M_{0,m+l})\setminus Fix_\theta)(-1)\to \]\[H^{a-1}(\mathcal{M}_{0,F(v')}\times \mathcal{M}_{0,F(v'_{\theta})})(-1)\otimes 
 H^{b-1}(\mathcal{M}_{0,F(v'')}\times \mathcal{M}_{0,F(v''_{\theta})})(-1),\]
which is obtained from equation\, \ref{E:3} by using the K\"unneth formula, a Tate twist $(-1)$ and multiplying by the Koszul sign $(-1)^{a-1}$. Finally, using the Poincar\'e duality, leads to the definition of a composition product $\circ_{i}$ for $Grav_{NY}$.

It thus remains to verify if the $Grav_{NY}$ satisfies the equivariance and associativity axioms. 
The equivariance axiom is, indeed, not modified. The main argument is that $t_i\in \Ss_{i}$ acts on the set of marked points $\{1,...,i\}$ and, thus on the set obtained under $\theta$. So, the following diagram remains commutative: 
\begin{center}
\begin{tikzpicture}[->,>=stealth',shorten >=1pt,auto,node distance=2cm,main node/.style={font=\sffamily\bfseries}]
\node[main node] (1) {$\Ps(k)\otimes \Ps(r_1)\otimes \dots \otimes \Ps(r_k)$};
 \node[main node] (2) [right of=1]{};
 \node[main node] (3) [right of=2]{};
 \node[main node] (4) [right of=3]{};
 \node[main node] (5) [right of=4] {$ \Ps(k)\otimes \Ps(r_1)\otimes \dots \otimes \Ps(r_k)$};
\node[main node] (6)[below of=1] {$ \Ps(r_1+....+r_k) $};
\node[main node] (7)[below of=5] {$ \Ps(r_1+....+r_k) $};

 \path[every node/.style={font=\sffamily\small}]
   (1) edge node[above] {$ id\otimes(t_1\otimes ...\otimes t_k)$} (5)
    edge node [left] {$\mu$} (6) 
   (5)edge node[right] {$ \mu$} (7)
	(6) edge node[below] {$t_1 \oplus...\oplus t_k$} (7) ;
\end{tikzpicture}.
\end{center}

The associativity axioms of a cyclic operad also holds. Therefore, on $Grav_{NY}$ there exists a natural cyclic operad structure.
\end{proof}

\begin{remark}
An explicit presentation of $Grav_{NY}(n)$, can be given by using the interpretation 
 \[Grav_{NY}(n)=s(H^{*}((\PC^{2n+3}\setminus (\Delta\bigcup_{i=1}^{n} \{x_i= \theta({x_i})=\frac{1}{2}\}) / PSL_2(\C)), \]
where $\Delta$ is the fat diagonal given by the colliding points i.e. $\{x_i=x_j\}$.

Then, it is enough to apply the works of~\cite{{DGM},{EPY},{Yuz99}}, to compute the cohomology of linear spaces minus hyperplanes.
\end{remark}
 \smallskip
 \begin{example}
We give an example on the construction of the $Grav_{NY}(n)$ operad . Consider the stratum of codimension 1. It is isomorphic to a product $\mathcal{M}_{0,2r+3}\times \mathcal{M}_{0,2s+3}$ with $(2r+2)+(2s+2)=2n+3$. Using the residue morphisms, we have that:
 \[H^{*}(\mathcal{M}_{0,2n+3}\setminus \cup_{i=1}^{n} H_i)\to H^{*-1}(\mathcal{M}_{0,2r+3}\setminus \cup_{i=1}^{r} H_i\, \times\, \mathcal{M}_{0,2s+3}\setminus \cup_{i=1}^{s} H_i )\] 
 where $H_i=\{x_i=\theta(x_i)=\frac{1}{2}\}$ are hyperplanes.
After dualization, the operad structure on the collection of graded vector spaces
$H_{*-1}(\mathcal{M}_{0,2n+3}\setminus H_i)$ is given.
\end{example}

\subsection{Comparison of the NY gravity operad with gravity operad and $Hycom$ operad}
We study the relations between the NY gravity operad and the gravity operad. Namely, a relation is given, using the comparison theorem in this section. 

The relations and the presentation for the Gravity operad are as follows. The degree 1 subspace of $Grav(k)$ is one dimensional for each $k\geq 2$, and is spanned by the operation: \[\{a_1,...,a_n\}=\sum_{1\geq i<j \geq n}(-1)^{\epsilon(i,j)}\{a_i,a_j\}a_1....\hat{a}_i....\hat{a}_j....a_n,\]
 $\epsilon(i,j)=(|a_1|+...+|a_{i-1}|)|a_i|+(|a_1|+...+|a_{j-1}|)|a_j|+|a_i||a_j|.$
Using the presentation of~\cite{Get}, the operations $\{a_1,...,a_k\}$ generate the gravity operad $Grav(k)$. All relations among them follow from the generalized Jacobi identity :
 \begin{equation}\label{E:2}
\begin{aligned}
\sum_{1\leq i<j\leq k}(-1)^{\epsilon(i,j)}\{\{a_i,a_j\}, a_1,...,\hat{a}_i,...,\hat{a}_j, \dots, a_k,&b_1,...,b_l\} \\
& =\{\{a_1,...,a_k\},b_1,...,b_l\},
\end{aligned}
\end{equation}
where the right hand side is 0 if $l=0$.


The gravity operad is related to the hypercommutative operad (in short $Hycom$ operad) by Koszul duality, in the sense of~\cite{Kap}, i.e. we have $Hycom^{!}\cong Grav$. By definition,
\[ 
Hycom(n)=\begin{cases} H_{*}(\M_{0,n}),& n \geq 3; \\ 0,& n < 3.\end{cases}
\]
 Supposing that $V\subset Hycom$ is the cyclic $\mathbb{S}$-submodule spanned by the fundamental classes. Then, 
\[[\M_{0,n}]\in H_{2(n-3)}(\M_{0,n})\subset Hycom(n).\]
The operad $Hycom$ is Koszul with generators $V$.


The properties of $Hycom$ are that it is quadratic with generators $\mathcal{V}$ and relations $R$, where $\mathcal{V}((n))$ is spanned by an element of degree $2(n-3)$ and weight $2(3-n)$. We have that $\mathcal{V}((n))$ is identified with $H_{2(n-3)}(\M_{0,n})$. Relations $R$---where $R((n))$ has dimension $n-1\choose 2$-1---are given by the following generalized associativity equation:
\begin{equation}
\sum_{S_1\sqcup S_2}\pm((a,b,x_{S_1}),c,x_{S_2})=\sum_{S_1\sqcup S_2}\pm(a,(b,c,x_{S_1}),x_{S_2}),
\end{equation} 
where $a,b,c,x_{1},...,x_{n}$ lie in $Hycom$ and $S_1\sqcup S_2=\{1,...,n\}$. The symbol $\pm$ stands for the Quillen sign convention for $ \mathbb{Z}_2$-graded vector spaces: it equals +1 if all the variables are of even degree.

Now, we can present the {\it Comparison theorem}.
\begin{theorem}[Comparison theorem]
Consider the $NY$ Gravity operad $Grav_{NY}$. 
Then, for $n\geq 3$ we have:
\[\begin{aligned}
Grav_{NY}(n)&=sH_{*}(\mathcal{M}_{0,n}\setminus Fix_{\theta})\otimes H_{*}(\mathrm{Aut}(\mathcal{M}_{0,n})\setminus Fix_{\theta})\\
\end{aligned}\]
where $ Fix_{\theta}$ is the set of fixed points of the automorphism $\theta$. 
\end{theorem}

\begin{proof}
To compare the gravity operad and the NY Gravity operad, notice that it is necessary to remove the fixed points under the involution $\theta$ (one needs smooth strata). Removing $Fix_{\theta}$ from $\mathcal{M}_{0,n}(\theta)$ leaves us to consider the smooth stratum of this space, and implies a modification regarding the classical $Grav$ operad.  
By the K\"unneth formula we have:

\[\begin{aligned}
Grav_{NY}(n)&=sH_{*}(\mathcal{M}_{0,n}\times \mathrm{Aut}(\mathcal{M}_{0,n})\setminus Fix_{\theta})\\
&=s(H_{*}(\mathcal{M}_{0,n}\setminus Fix_{\theta})\otimes H_{*}(\mathrm{Aut}(\mathcal{M}_{0,n})\setminus Fix_{\theta}))\\
&=sH_{*}(\mathcal{M}_{0,n}\setminus Fix_{\theta})\otimes H_{*}(\mathrm{Aut}(\mathcal{M}_{0,n})\setminus Fix_{\theta})
\end{aligned}\]
\end{proof}
 
\begin{remark}\,

We compare the NY operads' components to those of the gravity operad. This is done using the {\it excision theorem} and Mayer--Vietoris' exact sequence. We rely on notations of the previous examples. It follows from the excision theorem that there exists a relative homology of\, $H_{k}(\PC^{2n+3}\setminus \cup H_i,\, \mathcal{M}_{0,2n+3}\setminus \cup_{i=1}^{n} H_i)$ to $H_{k}(\PC^{2n+3},\, \mathcal{M}_{0,2n+3})$. 

 \smallskip
 
The excision theorem leads to the Mayer--Vietoris (long) exact sequence. Consider $\mathcal{M}_{0,2n+3}$ and let $A,B$ be subsets of $\mathcal{M}_{0,2n+3}$. We have 

\begin{itemize}
\item $A=\mathcal{M}^{\frac{1}{2}}_{0,2n+3}$ the moduli space $\mathcal{M}_{0,2n+3}$ from which is removed the hyperplane arrangement $\cup_{i=1}^{n} H_i$ 
\item $B$ is the tubular neighborhood of the hyperplane arrangement $\cup_{i=1}^{n}H_i$, denoted $Tub$.
\item The interiors of $A$ and $B$ cover $\mathcal{M}_{0,2n+3}$.
 \end{itemize}
Now, the long exact sequence gives:
\[\begin{aligned}
\dots \to H_{k+1}(\mathcal{M}_{0,2n+3})\to H_{k}&(Tub\setminus \cup_{i=1}^{n} H_i)\to H_{k}(\mathcal{M}^{\frac{1}{2}}_{0,2n+3})\oplus H_{k}(Tub)\\
& \to H_{k}(\mathcal{M}_{0,2n+3})\to H_{k-1}(Tub\setminus \cup_{i=1}^{n} H_i)\to \dots
\end{aligned}\]
 
Applying the suspension onto the long exact sequence, we have the following: 
\[\begin{aligned}
\dots\to sH_{k+1}(\mathcal{M}_{0,2n+3})\to sH_{k}&(Tub\setminus \cup_{i=1}^{n} H_i)\to sH_{k}((\mathcal{M}^{\frac{1}{2}}_{0,2n+3})\oplus sH_{k}(Tub)\\
& \to sH_{k}(\mathcal{M}_{0,2n+3})\to sH_{k-1}(Tub\setminus \cup_{i=1}^{n} H_i)\to \dots
\end{aligned}\]

and the long exact sequence looks in the following way: 
\[\begin{aligned}
\dots \to H_{2n+4}(\mathcal{M}_{0,2n+3})\to 0 \to H_{2n+3}((\mathcal{M}&^{\frac{1}{2}}_{0,2n+3})\to
H_{2n+2}(\mathcal{M}_{0,n}) \to \mathbb{Z}\to \\
 H_{2n+2}((\mathcal{M}^{\frac{1}{2}}_{0,2n+3}) 
&\to H_{2n+2}(\mathcal{M}_{0,2n+3})\to 0\to \dots
\end{aligned}\]
\end{remark}

\section{The neighbourhood of the fat diagonal and Hodge structure} 
Previously, we have introduced the NY Gravity operad, which we think describes an interesting geometric construction modeling an object coming from physics origins. In algebraic language all the physical requirements mean that to describe the geometry we need a specific algebra. Here we have chosen the algebra of split quaternion numbers, which fits nicely to space-time related problems, relativity and gravity.

The presence of supplementary singular points attracts our attention to how the geometry is perturbed. Those singular points appear in the realisation of the module over split quaternion numbers and in particular when it occurs to the 0-divisors of the underlying ring.


We propose to look at the neighbourhood of the fat diagonal for $\mathcal{M}_{0,n}(\theta)$ from a Hodge structure aspect. Our NY gravity operad being directly related to mixed Hodge structures, it is natural to consider the Hodge structure in a mostly geometric way, via tools coming from $L^2$-cohomology, and which are closely based on results concerning the KZ equations for $\mathcal{M}_{0,n} (\C)$ (see~\cite{Kap}), and to the Stasheff polytope approach~\cite{Ka}. 


\subsection*{The fat diagonal of $\mathcal{M}_{0,n}$ with hidden symmetry}
Let $M$ be a smooth manifold.
Let \[ [0,1]\times M /\sim,\] where $(0,x) \sim (0,x')$ and $x, x'\in M$ be a (topological) cone. 
Let us stratify the thick diagonal in the following (rather unconventional) way: 
\begin{equation}\label{E:D}\Delta^{(n)} \subset \Delta^{(n-1)} \subset \dots\subset\Delta^{(1)},\end{equation} where $\Delta^{(1)}$ is the $n$--tuple of $\Delta$, where there exists one $x_i$ in an $\epsilon$-neighborhood of $\frac{1}{2}$, different from all the other $x_j$ in $\Delta$;
$\Delta^{(j)}$ is the $n$--tuple, where $\{x_{i_{1}}=\dots =x_{i_{j}}\}$ are in an $\epsilon$-neighborhood of $\frac{1}{2}$, and those $x_{i}$ are different from the remaining $x_k$ in $\Delta$.

\begin{proposition}
Let $\Delta_{\theta}$ be the fat diagonal of $\M_{0,n}(\theta)$. 
Then, $\Delta_{\theta}$ has a $\mathbb{Z}_2$-symmetry about the locus in $\Delta_{\theta}$ corresponding to $Fix_{\theta}$. Locally, the locus $Fix_{\theta}$ defines the apex of a cone which can be decomposed into two symmetrical subcones. 
\end{proposition}
\begin{proof}
The space $\M_{0,n}(\theta)$ being equipped with the split quaternions structure, it subdivides into a pair of isomorphic (and symmetric) subspaces. The singular locus (i.e. fat diagonal) of this spaces obeys naturally to the same property. Therefore, $\Delta_{\theta}$ is divided into symmetrical subspaces, which we call respectively $\Delta_{+}$ and $\Delta_{-}$
Consider the stratum $\Delta^{(1)}$, defined in formulae~\ref{E:D}. 
Then, in the stratum $\Delta_{+}^1$ there exists one point $x_i$ tending to $Fix_{\theta}$. In $\Delta_{-}^{1}$, we have its symmetric version (recall examples of section~\ref{s:sdm}). So, in a small neighborhood of $Fix_{\theta}$, we have a topological cone of apex $Fix_{\theta}$ defined over the intersection of a small sphere with center in $Fix_{\theta}$ and $\Delta_{+}^1\cup \Delta_{-}^1$. By induction we can see that the same holds for strata of higher codimension. We have thus defined a cone of apex $Fix_{\theta}$ which can be sliced into a pair of symmetrical subcones. 
\end{proof}

The fat diagonal of the parametrizing scheme $\M_{0,n}\times \mathrm{Aut}(\M_{0,n})$ is endowed with additional critical points, compared to the classical case (i.e. with no symmetry). 
Going back to the stable curve, these additional points will be blown-up, adding new irreducible components (isomorphic to $\PC^1$). This data modifies the geometric aspect around the divisor of $\M_{0,n}\times \mathrm{Aut}(\M_{0,n})$, compared to the one of the stable curve $\M_{0,n}$. Indeed, those points, lying in $Fix_{\theta}$, are conjugated to complex germs of the type $(z^{2k},0)$, where $1\leq k\leq n$. Therefore, not only this modifies the stratification of the stable curves, (comparing it to the one of $\M_{0,n}$), but also implies differences in the metrics around the divisors.

\subsection{The neighborhood of the fat diagonal in $\mathcal{M}_{0,n}(\theta)$}


We consider the neighborhood of the fat diagonal in $\mathcal{M}_{0,n}(\theta)$. One approach is done using K\"ahler geometry; the second one using Riemannian geometry. Both approaches describe the split quaternionic structure and lead in the end to the cohomology ring of the investigated space. 
 
 \smallskip

Recall some notions. Let $Conf_{0,n}$ be the configuration space of n marked points on the complex plane. 
For $1\leq j \neq k \leq n$, let $w_{jk}=\frac{dlog(x_j-x_k)}{2\pi\imath}$ be the logarithmic differential form.
The cohomology ring $H^{*}(Conf_{0,n},\mathbb{Z})$ is the graded commutative ring with generators $[w_{jk}]$, and relations:

\begin{itemize}
\item $w_{jk}=w_{kj}$
\item $w_{ij}w_{jk}+w_{jk}w_{ki}+w_{ki}w_{ij}=0$
\end{itemize}

The cohomology ring $H^{*}(\mathcal{M}_{0,n+1}, \C)$
may be identified with the kernel of the differential $\imath$ on $H^{*}(Conf_{0,n},\C)$ whose action on the generators is $\imath w_{jk}=1$.

The K\"ahler logarithmic derivatives approach is of interest to us, because not only it is related to the KZ equations but also to $L^2$--cohomology and thus in a certain sense to the Hodge structure. We wish to highlight the bridges between those seemingly different view points. The generators $[w_{jk}]$ of the cohomology ring $H^{*}(Conf_{0,k+1})$ are induced by $w_{jk}=\frac{dlog(x_j-x_k)}{2\pi\imath}$ and those logarithmic derivatives are also a way to have some generators for the $L^2$--cohomology. 

An incomplete Riemannian metric over a given smooth manifold (for instance $\mathcal{M}_{0,n}$) can be obtained using a projective embedding. Recall its construction. If $p:\mathbb{C}^{n+1}-0\to \mathbb{P}^n$ is the projection where $(x_{1},\dots,x_{n})\in \mathbb{C}^{n+1}$, then 
$\mathbb{P}^n$ is endowed with a natural K\"ahler metric $\nu$, called Fubini-Study metric and defined by: 
\[p^{*}\nu=\frac{i}{2\pi}\partial\overline{\partial}log | |x_{1}|^2+ |x_{2}|^2 +\dots + |x_{n}|^2 |.\]
This gives a point-wise norm on smooth forms $\omega$ of type $(p,q)$ on $\mathcal{M}_{0,n}$, and defines an $L^{2}$ norm $||\omega||_{2}$.
One defines thus a simplicial complex by setting $\mathcal{F}^{p,q}(\mathcal{M}_{0,n}):=\{\omega | ||\omega||_{2}<\infty, ||\bar{\partial}\omega||_{2}<\infty\}$. A Dolbeault type of complex, $(\mathcal{F}^{p,*}(\mathcal{M}_{0,n+1}),\bar{\partial})$ for each $p \geq 0$ can be defined, from this definition. The existence of such a complex allows the definition of an $q$-th $L^{2}$-cohomology group. It is usually denoted by $H_{{2}}^{p,q}(\mathcal{M}_{0,n})$, in the literature. See for more details~\cite{{PS1},{PS2}}.

A subtile change of variables leads to define locally a K\"ahler metric on the smooth part of the moduli space $\mathcal{M}_{0,n+1}$, in the neighborhood of the fat diagonal $\{x_i=x_j| i\neq j\}$. This construction allows to define a norm with respect to the metric. Using this tool, we can apply results from~\cite{{DF},{D},{O}} to describe the $L^2$-cohomology. 
 
Choose a ball $B_{r}$, centered at 0 and of radius $r$, such that $\sum_{i=1}^{n}|x_{i}|^{2}\ll1$. Consider the Euclidean distance between a pair of points $x_i$ and $x_j$, where those points lie in the interior of the ball of radius $r$. 
Set 
\begin{equation}
F= -log(dist(x_i,x_j)),  
\end{equation}
and 
\begin{equation}
F_{k}:= -log(dist(x_i,x_j))-\frac{1}{k}log(-log\sum|x_{i}|^2)
\end{equation}
where $x_i,x_j\in B_{r}$ and $k>1$.

The latter strictly pluri-subharmonic function defines a K\"ahler metric $h_{k}:=-i\partial\bar{\partial}F_{k}$ on the regular part of $\mathcal{M}_{0,n+1}(\theta)\setminus Fix_{\theta}\cap B_{r}$. The metric $h_{k}:=-i\partial\bar{\partial}F_{k}$ on $\mathcal{M}_{0,n+1}(\theta)\setminus Fix_{\theta}$ is complete and decreases monotonically to $h:=-i\partial\bar{\partial}F$. Independently of $k$, $\langle \partial F_k,\partial F_k\rangle$ is bounded, where $\langle , \rangle_{k}^{\frac{1}{2}}$ denotes the pointwise norm on 1-forms with respect to $h_k$.

The following propositions got modified and updated regarding our explanations concerning the investigated structure
\begin{proposition}\label{P:PS1}
Let $\mathcal{M}_{0,n+1}(\theta)\setminus Fix_{\theta}$ be endowed with a K\"ahler metric, which is given by the potential function $F:\mathcal{M}_{0,n+1}(\theta)\setminus Fix_{\theta}\to \R$ such that 
$\langle \partial F,\partial F\rangle$ is bounded. Then, the $L_2-\overline{\partial}$-cohomology with respect to $\omega$, $H_{(2)}^{p,q}(\mathcal{M}_{0,n+1}(\theta)\setminus Fix_{\theta},\omega)=0$ for $p+q\neq n+1$. In fact, if $\langle \partial F,\partial F \rangle^{\frac{1}{2}}\leq B$, and $\phi$ is a $(p,q)$-form on $\mathcal{M}_{0,n+1}(\theta)\setminus Fix_{\theta}$ with $\overline{\partial}\phi=0$, $q>0$ and $p+q\neq 0$, then there is a $(p,q-1)$-form $\nu$ such that $\overline{\partial}\nu=\phi$ and $||\nu||\leq 4B||\phi||$.
\end{proposition}

 This statements given in Proposition~\ref{P:PS1} and Proposition~\ref{P:PS2} are a direct adaptation of Lemma 2.2 and Lemma 2.3 from~\cite{PS2} to our setting. Globally speaking, statements in Lemma (2.2) and (2.3) from~\cite{PS2} holds for any complex manifold. This is interesting for us here since $\mathcal{M}_{0,n+1}(\theta)\setminus Fix_{\theta}$ is still considered as complex variety although it is a realisation of the corresponding module over the split quaternion numbers. 

\begin{proposition}\label{P:PS2}
Let $\mathcal{M}_{0,n+1}(\theta)\setminus Fix_{\theta}$ be endowed with a decreasing sequence of complete hermitian metrics $h_k$, $k\geq 1$, which converges pointwise to a hermitian metric $h$. If $H_{(2)}^{n+1,q}(\mathcal{M}_{0,n+1}(\theta)\setminus Fix_{\theta},h_k)$ vanishes with an estimate that is independent of $k$, then $H_{(2)}^{n+1,q}(\mathcal{M}_{0,n+1}(\theta)\setminus Fix_{\theta},h)$ vanishes with an estimate. 
\end{proposition}
Here, $H_{(2)}^{n+1,q}(\mathcal{M}_{0,n+1}(\theta)\setminus Fix_{\theta},h_k)$ denotes $L_2-\overline{\partial}$-cohomology with respect to the metric $h_k$.

Hence, these results provide interesting results concerning the K\"ahler geometry for the realisation of the split quaternionic structure, in our context.
We now continue our investigations, but from an Euclidean geometry point of view. Note that the approach below holds no longer for the parametrizing scheme $T$, but for the universal curve, where colliding points have already been blown-up.

 \subsection{Riemannian geometry in the neighbourhood of the divisor}
In this subsection, given parametrizing scheme $T$, consider the universal curve and study the Riemannian geometry around the divisor $\Delta_{\theta}$. 
We highlight the geometric differences in terms of Riemannian geometry around the divisors of $\M_{0,n+1}$ and of $\M_{0,n+1}(\theta)$. The creation of additional singularities due to the split quaternionic structure model implies modifications around the geometry, which we want to understand better. 
 
To avoid any notation confusion: let $D^{\tau}$ be a given divisor of the $n+1$-stable curve $\M_{0,n+1}$, parametrized by a subscheme $D(\tau)$ in $T$. For this subscheme in $T$, let $D_{\tau}$ be the divisor of $NY(n+1)$. The notation $D$ is used for a normal crossing divisor.

The aim of the next proposition is to give an idea of the geometric differences occurring in $\M_{0,n+1}$ and $\M_{0,n+1} \times \mathrm{Aut}(\M_{0,n+1})$.
We introduce the classical filtration: $\Delta_{(0)}\subset \Delta_{(1)} \subset \dots \Delta_{(n)}$, where an $i$-th (smooth) stratum is given by $\Delta_{(i)} - \Delta_{(i-1)}$ (note that this is a different stratification than in the previous section). The analogous stratification is done for the fat diagonal $\Delta_{\theta}$ of $\M_{0,n+1} \times \mathrm{Aut}(\M_{0,n+1})$.

In order to understand the neighbourhood--geometry of $\Delta_{\theta}$ in $\M_{0,n}(\theta)$ we apply some results of Hsiang--Pati~\cite{HP85} and Nagase~\cite{Na89}. In the following part, we assume that we intersect the fat diagonal with a $n$-ball of radius 1 centered at 0. It is known that this intersection behaves as a topological cone defined over the boundary of this intersection.
We claim the following:

\begin{proposition}\label{P:a}
Let us stratify the singular locus $\Delta(\theta)$ as above. Let $D_+$ and $D_{-}$ be the connected components of $D_\theta$ symmetric to each other about the component $D_{Fix_{\theta}}$ corresponding to $Fix_{\theta}$. 
Suppose that $(r,\Theta,y) \in (0,1]\times[0,1]\times Y$ and $\tilde{g}(y)$ is Riemannian metric on a manifold $Y$. 
Then, on each stratum of $\Delta_{\theta}$ there exists an isolated (conical) singularity of type $A_k$. For a suitable system of coordinates, we describe the Riemannian metric in the neighbourhood of $D_{\theta}$ as follows:
\[ 
\begin{cases}
\text{The neighborhoof of}\quad D_+: & dr^2+r^2d\Theta^2+r^{2c}\tilde{g}(y). \\ 
\text{The neighborhoof of}\quad D_-: & dr^2+r^2d\Theta^2+r^{2c}(ds^2+h^{2}(r,s)d\phi), \\
& h\, \text{is a smooth function.} \\
\text{The neighborhoof of}\quad D_{Fix_{\theta}}: & dr^2+r^2d\Theta^2+r^{2}\tilde{g}(y)
 \end{cases}\]
\end{proposition}
\begin{remark}
In the statement above, the notation for the Riemannian metric $\tilde{g}(y)$ is used in a large sense, i.e. depends completely on the neighbourhood of the divisor. 
\end{remark}


 For the reader's convenience, let us recall the construction of Hsiang--Pati and Nagase metric around $D$ as in \cite{HP85,Na89}. Roughly speaking the Hsiang--Pati/ Nagase metric is of the type $dr^2+r^2d\Theta^2+r^{2c}\tilde{g}(y)$ where $(r,\Theta,y) \in (0,1]\times[0,1]\times Y$ and $\tilde{g}(y)$ is Riemannian metric. However, we distinguish the Hsiang--Pati metric from the Nagase, for the following technical reason. The Hsiang--Pati metric is defined around the divisor's components (but not around intersections of the divisors' components) and the Nagase metric is meant to be used on the remain complementary part (i.e. around the intersection of the divisors' components).




The metrics are given as follows. 
\begin{enumerate}
\item Let $Y$ be a compact polygon in $ \mathbb{R}^2$ with standard metric $\tilde{g}$.\\
Let $W_{HP}=(0,1]\times[0,1]\times Y \ni (r,\Theta,y)$, be endowed with the Riemannian metric: \[g_{HP}:=dr^2+r^2d\Theta^2+r^{2c}\tilde{g}(y).\]
\item Let $W_{N}=(0,1]\times [0,1]^3$ be endowed with the Riemannian metric: \[g_{N}:=dr^2+r^2d\Theta^2+r^{2c}(ds^2+h^{2}(r,s)d\phi),\]
with $h(r,s)=\frac{f(r)}{l(\frac{s}{f(r)})}$, where $f(r)$ is a smooth function on $[0,1]$ such that $f'(r)\geq 0, \forall r\geq 0$ and $l(x)$ is a smooth function on $[0,\infty)$ such that $l'(x)\geq 0$ and $l''(x)\geq 0$ for any $x\geq0$. 
We define these two functions as : 

\begin{itemize}
\item $f(r)=r^b$, if $r$ is small and $r>0, b>0$.
\item $f(r)=\frac{1}{2}$ if $r$ is large and $r\leq 1$. 
\end{itemize}

\begin{itemize}
\item $l(x)=1$, if $0\leq x\leq 1-\epsilon$.
\item$l(x)=x$ if $1+\epsilon \leq x$.
\end{itemize}

\end{enumerate}

The $W_{HP}$ and $W_{N}$ are models obtained up to quasi-isometry of the subsets defining a tubular neighborhood around the divisor.
 By quasi-isometry we mean that for Riemannian manifolds $(Y_{i},g_{i})$, the diffeomorphism $f:(Y_{1},g_{1})\to(Y_{2},g_{2})$ satisfies for a positive constant ${\it a}>0$ the inequality: $C^{-1}g_{1}\leq f^{*}g_{2} \leq {\it a} g_{1}$.

\smallskip 
We can now conclude with the proof of Proposition~\ref{P:a}.
\begin{proof}
Using the classical stratification for $\Delta_{\theta}$ we may notice that on each stratum, there exists a point lying in $Fix_{\theta}$. In particular, for each codimension $k$ stratum there exists an isolated singularity of type $A_{k-1}$. Blowing-up this point, we can describe the geometry around its irreducible components, using the Hsiang--Pati and Nagase metrics. Now, following some technical computations, we find out that the coefficient for the metric defined around the locus $D_{Fix_{\theta}}$ is $c=1$. 
 \end{proof}
 

\end{document}